\documentclass[12pt, a4]{amsart}
\usepackage{amsgen,amsmath,amstext,amsbsy,amsopn,amsfonts,amssymb}
\usepackage{amsmath,amssymb,amsfonts,amsthm,enumerate}
\usepackage{amsthm}
\usepackage[title]{appendix}
\usepackage[utf8]{inputenc}
\usepackage{color}
\usepackage[pagebackref]{hyperref}
\usepackage[alphabetic]{amsrefs}
\usepackage{hyperref}
\usepackage{xcolor}

\newtheorem{theorem}{Theorem}
\newtheorem{lemma}[theorem]{Lemma}

\newtheorem{corollary}[theorem]{Corollary}

\newtheorem{obs}[theorem]{Observation} \newtheorem{defi}[theorem]{Definition}

\newenvironment{definition}{\begin{defi}\rm}{\end{defi}}
\newtheorem{exa}[theorem]{Example}

\newtheorem{rem}[theorem]{Remark}
\newenvironment{remark}{\begin{rem}\rm}{\end{rem}}
\newtheorem{rems}[theorem]{Remarks}

\newtheorem{ack}[theorem]{Acknowlegment}

\def\H{\mathcal H}

\def\L{\mathcal L}
\def\K{\mathcal K}

\def\D{\mathcal D}

\def\Z{\mathcal Z}

\def\NN{{\mathbf N}}
\def\ZZ{{\mathbf Z}}
\def\CCC{{\mathbf C}}

\def\QQ{\mathbf Q}
\def\FF{\mathbf F}

\def\RR+{{\mathbf R}^*}
\def\TT{\mathbf T}

\def\kk{\mathbf k}

\def\Q_p{{\mathbf Q}_p}

\def\Proj{\rm Proj}

\def\ind{{\rm Ind}}

\def\Ga{\Gamma}
\def\ga{\gamma}
\def\La{\Lambda}
\def\la{\lambda}

\def\vfi{\varphi}

\def\Aut{{\rm Aut}}

\def\ind{{\rm Ind}}

\def\Ad{{\rm Ad}}

\newcommand{\GL}{\operatorname{GL}}

\newcommand{\Cent}{\operatorname{Cent}}
\newcommand{\Tr}{\operatorname{Tr}}

\newcommand{\Char}{\operatorname{Ch}}
\newcommand{\Ind}{\operatorname{Ind}}

\def\tout{\qquad\text{for all}\quad}
\newcommand{\Hi}{\mathcal H}

\newcommand{\Ki}{\mathcal K}

\newcommand{\GaFC}{\Ga_{\rm fc}}
\newtheorem{thmx}{Theorem}
 
\newtheorem{corx}[thmx]{Corollary}
\date{\today}
\begin{document}

\title[Plancherel formula for countable groups]{The Plancherel formula for  countable groups}

\address{Bachir Bekka \\ Univ Rennes \\ CNRS, IRMAR--UMR 6625\\
Campus Beaulieu\\ F-35042  Rennes Cedex\\
 France}
\email{bachir.bekka@univ-rennes1.fr}

\author{Bachir Bekka}

\thanks{The author acknowledges the support  by the ANR (French Agence Nationale de la Recherche)
through the project Labex Lebesgue (ANR-11-LABX-0020-01) .}
\begin{abstract}
We discuss a Plancherel formula for   countable groups,
which  provides a canonical   decomposition of the regular representation 
 of such a group $\Gamma$ into a direct integral of factor representations. 
Our main result  gives a precise description of this decomposition
in terms of the Plancherel formula  of the FC-center $\Gamma_{\rm fc}$ of $\Gamma$ (that is, the normal sugbroup 
of $\Ga$ consisting of elements with a finite conjugacy class); this description involves the action
of an appropriate totally disconnected compact group of automorphisms of $\Gamma_{\rm fc}$.
As an application, we determine the Plancherel formula for linear groups.  
In an appendix,  we use the Plancherel formula to provide a  unified  proof  for Thoma's and Kaniuth's theorems which respectively  characterize countable groups which are of type I and those whose regular representation is of type II.

\end{abstract}
\maketitle
\section{Introduction}
\label{Intro}
Given a  second countable locally compact group $G,$ a fundamental 
object to study is its \textbf{unitary dual space} $\widehat{G}$,
that is,   the set irreducible unitary representations  of $G$ up to unitary equivalence.  
The space $\widehat{G}$ carries a natural Borel structure, called the Mackey Borel structure  (see \cite[\S 6]{Mackey-Borel} or \cite[\S 18.6]{Dixm--C*}).
A classification of  $\widehat{G}$ is considered as being  possible only if  $\widehat{G}$ 
is a standard Borel space;  according to  Glimm's celebrated theorem (\cite{Glimm}), this is the case if and only if  $G$ is of type I in the following sense.

Recall that a von Neumann algebra  is a self-adjoint subalgebra of $\L(\H)$ which is closed for the weak operator topology of $\L(\H)$, where $\H$  is a  Hilbert space.
 A von Neumann algebra  is a factor if  its center  only consists  of the scalar operators.
 
Let  $\pi$ be a  unitary representation of $G$ in a Hilbert space $\H$
(as we will only consider  representations which are unitary, we will often drop the
adjective ``unitary").
The von Neumann subalgebra of  $\L(\H)$ generated by $\pi(G)$  coincides
with the bicommutant $\pi(G)''$  of $\pi(G)$ in $\L(\H)$; we say that
$\pi$ is a factor representation if $\pi(G)''$   is a factor.

\begin{definition}
\label{Def-TypI}
The group  $G$ is of \textbf{type I} if, for every  factor representation
$\pi$ of $G$, the factor $\pi(G)''$ is of type I, that is, $\pi(G)''$ is isomorphic
to the von Neumann {algebra} $\L(\K)$ for some Hilbert space $\K;$ equivalently, 
the Hilbert space $\H$ of $\pi$ can written as 
tensor product $\K\otimes \K'$ of Hilbert spaces 
in such a way that 
$\pi$ is equivalent to  $\sigma\otimes I_{\K'}$ for an irreducible representation $\sigma$ of $G$ on $\K.$
\end{definition}

Important classes of groups are known to be of type I, such as semi-simple 
or nilpotent Lie groups.
A major problem in harmonic analysis is to  decompose  the left  regular representation  $\la_G$ on 
$L^2(G, \mu_G)$ for a Haar measure $\mu_G$
as a direct integral of irreducible representations.
When $G$ is of type I and unimodular, this is the content of the  classical 
 Plancherel theorem: there exist a unique measure $\mu$ on $\widehat{G}$ and
a unitary isomorphism between $L^2(G,\mu_G)$ and the direct integral of Hilbert spaces 
$\int_{\widehat{G}}^\oplus (\H_\pi\otimes \overline{\H_\pi}) d \mu(\pi)$  which transforms 
$\la_G$  into $\int_{\widehat{G}}^\oplus (\pi \otimes  I_{\overline{\H_\pi}})d\mu(x)$,
where $\overline{\Hi_\pi}$ is the conjugate of the Hilbert space $\Hi_\pi$ of $\pi;$
in particular, we have a Plancherel formula
$$
\Vert f\Vert_2^2 =\int_{\widehat{G}} \Tr(\pi(f)^*\pi(f)) d\mu(\pi)
\tout f\in L^1(G, \mu_G)\cap L^2(G, \mu_G),
$$
 where $\Vert f\Vert_2$ is the $L^2$-norm of $f,$  $\pi(f)$
 is the  value at $f$ of the natural extension of $\pi$ to a representation of $L^1(G, \mu_G)$, and 
$\Tr$ denotes the standard trace on $\L(\H_\pi);$ for all this,  see \cite[18.8.1]{Dixm--C*}.

When $G$ is  not type I, $\la_G$   usually admits several integral decompositions into irreducible representations
and it is not possible to single out a canonical one among them. 
However, when $G$ is unimodular, $\la_G$ does admit a canonical integral decomposition into  factor representations 
representations;  this the content of a  Plancherel theorem  which we will discuss in the case of a discrete group
(see Theorem~\ref{Pro-Plancherel}).

Let $\Ga$ be a countable group.
 As discussed below (see Theorem~\ref{Theorem-Thoma-Kaniuth}), $\Ga$ is  usually not of type I. 
 In order to state the Plancherel theorem for $\Ga,$ we need to replace
 the dual space $\widehat{\Ga}$ by the consideration 
 of Thoma's dual space $\Char(\Ga)$ which we now introduce.

 Recall that a function $t:\Ga\to \CCC$
is of positive type if the complex-valued matrix $(t(\ga_j^{-1} \ga_i))_{1\leq i,j\leq n}$ is positive semi-definite for any $\ga_1, \dots, \ga_n$ in $\Ga$

A function of positive type $t$ on $\Ga$ which is constant
on conjugacy classes and normalized (that is, $t(e)=1$) will be called a \textbf{trace}
on $G$. The set of traces on $\Ga$ will be denoted by  $\Tr(\Ga)$.

Let $t\in  \Tr(\Ga)$ and $(\pi_t, \H_t, \xi_t)$ be the  associated GNS triple 
(see \cite[C.4]{BHV}). Then 
$\tau_t: \pi(\Ga)''\to \CCC,$ defined by 
$\tau_t(T)=  \langle T \xi_t\mid \xi_t\rangle $ is a  trace on 
the von Neumann algebra $\pi_t(\Ga)'',$ 
that is, $\tau_t (T^*T)\geq 0$ and  $\tau_t(TS)=\tau_t(ST)$ for all $T,S\in \pi_t(\Ga)'';$
moreover, $\tau_t$ is faithful in the sense that $\tau_t(T^*T)>0$
for every $T\in \pi_t(\Ga)'', T\neq 0.$ Observe that 
$\tau_t(\pi(f))= t(f)$ for $f\in \CCC[\Ga],$ where $t$
denotes  the linear extension
of $t$ to the group algebra $\CCC[\Ga].$

The set $\Tr(\Ga)$  is  a 
convex subset of the unit ball of $\ell^\infty (\Ga)$ which is 
compact in the topology of pointwise convergence.
An extreme point of  $\Tr(\Ga)$  is called  	a \textbf{character} of $\Ga$;
we will refer to  $\Char(\Ga)$ as  \textbf{Thoma's dual space}.

Since $\Ga$ is countable, $\Tr(\Ga)$ is a compact metrizable space and 
$\Char(\Ga)$  is easily seen to be  a $G_\delta$ subset   of $\Tr(\Ga)$.
So, in contrast to $\widehat{\Ga}$, Thoma's dual space $\Char(\Ga)$ is always 
a standard Borel space.

An important fact is that $ \Tr(\Ga)$ is a simplex (see \cite[Satz 1]{Thoma1} or \cite[3.1.18]{Sakai});  by Choquet theory, this implies that every $\tau\in \Tr(\Ga)$ 
can be represented as integral $\tau=\int_{\Char(\Ga)}t d\mu(t)$
for a \emph{unique} probability measure $\mu$ on $\Char(\Ga).$

As we now explain, the set of characters of $\Ga$ parametrizes the factor representations of  finite type of $\Ga,$ up to quasi-equivalence; for more details, see \cite[\S 17.3]{Dixm--C*} or \cite[\S 11.C]{BH}.

Recall first that two  representations 
$\pi_1$ and $\pi_2$ of $\Ga$ are quasi-equivalent if there exists
an isomorphism $\Phi: \pi_1(\Ga)''\to \pi_2(\Ga)''$ of von Neumann algebras
such that $\Phi(\pi_1(\ga))=\pi_2(\ga)$ for every $\ga\in \Ga.$

Let $t\in  \Char(\Ga)$ and $\pi_t$ the associated GNS representation. Then 
 $\pi_t(\Ga)''$ is a  factor of finite type.
Conversely, let $\pi$ be a representation of $\Ga$ such that $\pi(\Ga)''$ is a  factor of finite type and let $\tau$ be the unique
normalized trace on $\pi(\Ga)''.$
 Then  $t:=\tau \circ \pi$ belongs to   $\Char(\Ga)$
 and only depends on the quasi-equivalence class $[\pi]$ of $\pi.$
 
 The map  $t\to [\pi_t]$ is a bijection between $\Char(\Ga)$ and the 
 set of quasi-equivalence classes of factor representations of  finite type of $\Ga.$

The following result   is a version  for countable groups of  a   Plancherel theorem 
due to Mautner \cite{Mautner} and Segal \cite{Segal} which holds more generally for 
any unimodular second countable  locally compact group;  its proof is easier 
in our setting and will be given in   Section~\ref{ProofPL}  for the  convenience of the reader.
\begin{thmx}
\label{Pro-Plancherel}
(\textbf{Plancherel theorem for countable groups})
Let $\Gamma$ be a countable group.
There exists a  probability measure $\mu$ on $\Char(\Ga)$,
a measurable field of representations $t \mapsto (\pi_t, \H_t)$ of $\Ga$
on  the standard Borel space $\Char(\Ga)$,
and an isomorphism of Hilbert spaces between $\ell^2(\Ga)$ and $\int_{\Char(\Ga)}^\oplus \H_t d\mu(t)$
which transforms $\la_\Ga$ into $ \int^\oplus_{\Char(\Ga)} \pi_t d\mu(t)$ and has  the following properties:
\begin{itemize}
\item[(i)]
$\pi_t$ is quasi-equivalent to the GNS representation associated to $t$;
in particular, the $\pi_t$'s are mutually disjoint factor representations of finite type,
for $\mu$-almost every $t\in\Char(\Ga)$;
\item[(ii)] the von Neumann algebra $L(\Ga):=\la_\Ga(\Ga)''$ is  mapped \emph{onto}  the direct integral
 $\int_{\Char(\Ga)}^\oplus \pi_t(\Ga)'' d\nu(t)$ of factors;
\item[(iii)] for every $ f\in \CCC[\Ga],$ the following  Plancherel formula holds:
$$
\Vert f\Vert_2^2 =\int_{\Char(\Ga)} \tau_t(\pi_t(f)^*\pi_t(f))d\mu(t)= \int_{\Char(\Ga)} t(f^*\ast f)d\mu(t).
$$
\end{itemize}
The measure $\mu$ is the unique probability measure on $\Char(\Ga)$ 
such that the Plancherel formula above holds. 
\end{thmx}
The probability measure $\mu$ on $\Char(\Ga)$ from Theorem~\ref{Pro-Plancherel}   is called the \textbf{Plancherel measure} of $\Ga$.

\begin{remark}
\label{remark-PlancherelDiscrete}
The Plancherel measure gives rise  to what seems to be an interesting dynamical system on  $\Char(\Ga)$ involving 
the group $\Aut(\Ga)$  of automorphisms of $\Ga.$
We  will equip $\Aut(\Ga)$ with 
 the topology of pointwise convergence on $\Ga,$ for which it is a totally disconnected topological group. 
 The natural action of $\Aut(\Ga)$ on $\Char(\Ga)$, given
by $t^g(\ga)= t(g^{-1}(\ga))$ for $g\in \Aut(\Ga)$ and $t\in \Char(\Ga),$ is clearly continuous.

Since the induced action of $\Aut(\Ga)$ on $\ell^2(\Ga)$ is isometric, the following fact is an immediate consequence of
the uniqueness of the Plancherel measure $\mu$ of $\Ga:$

\medskip
\centerline{ the action of $\Aut(\Ga)$ on  $\Char(\Ga)$ preserves  $\mu$.}
\medskip

For example,  when $\Ga=\ZZ^d$, Thoma's dual 
$\Char(\Ga)$ is the torus $\TT^d$, the Plancherel measure  $\mu$ is the normalized Lebesgue
measure on $\TT^d$ which is indeed preserved by the group $\Aut(\ZZ^d)=\GL_d(\ZZ)$. Dynamical systems 
of the form $(\Lambda, \Char(\Ga), \mu)$ for a subgroup $\Lambda$ of $\Aut(\Ga)$ 
  may be viewed  as generalizations of  this example.
 
\end{remark}

 We denote by  $\Ga_{\rm fc}$ the  FC-centre of $\Ga,$  that is, 
the normal subgroup of elements in $\Ga$  with a finite conjugacy class.
It turns out (see Remark~\ref{SupportPlancherel})  that $t=0$ on $ \Ga\setminus \Ga_{\rm fc}$ for $\mu$-almost every $t\in\Char(\Ga)$.
In particular,  when $\Ga$ is ICC, that is, when $\GaFC=\{e\},$ 
the regular representation $\la_\Ga$ is factorial (see also Corollary~\ref{Cor-Lem-Center})
so that the Plancherel formula is vacuous in this case.

In fact, as we now see, the  Plancherel measure
of $\Ga$ is entirely determined by the Plancherel measure of $\Ga_{\rm fc}.$
Roughly speaking, we will see that 
 the Plancherel measure of $\Ga$ is the image of the  Plancherel measure of $\GaFC$ under the quotient map $\Char(\GaFC)\to \Char(\GaFC)/K_\Ga,$
 for a compact group $K_\Ga$ which we now define.

Let $K_\Ga$ be the closure
in $\Aut(\GaFC)$ of the subgroup $\Ad(\Ga)|_{\GaFC}$ 
given by conjugation with elements from $\Ga.$
Since every $\Ga$-conjugation class in $\GaFC$ is finite,
$K_\Ga$ is a compact group.
By a general fact about actions of compact groups on Borel spaces
(see Corollary 2.1.21 and Appendix in \cite{Zimmer}),
the quotient space 
$\Char(\GaFC)/K_\Ga$ is a standard Borel space.

Given a  function $t: H\to \CCC$ of positive type on a subgroup $H$ of a group $\Ga,$ we denote by  $\widetilde{t}$ the extension  of $t$ to $\Ga$ given by $\widetilde{t}=0$ outside $H.$ Observe that  $\widetilde{t}$ is of positive type on $\Ga$ (see for instance \cite[1.F.10]{BH}).

Here is our main result.
 \begin{thmx}
 \label{Pro-PlancherelFC}
 (\textbf{Plancherel measure: reduction to the FC-center})
 Let $\Gamma$ be a countable group. Let $\nu$ be
the  Plancherel measure  of $\Ga_{\rm fc}$
and $\la_{\GaFC}=\int^\oplus_{\Char(\GaFC)} \pi_t d\nu(t)$
the integral decomposition of the regular representation of $\GaFC$ as in 
Theorem~\ref{Pro-Plancherel}. Let 
$\dot{\nu}$ be the image of $\nu$ under the quotient map 
$\Char(\GaFC)\to \Char(\GaFC)/K_\Ga.$
\begin{itemize}
\item[(i)] 
 For every  $K_\Ga$-orbit $\mathcal O$ in $\Char(\GaFC),$ let 
 $m_{\mathcal O}$ be the unique normalized $K_\Ga$-invariant probability measure
 on $\mathcal O$  and let 
 $\pi_{\mathcal O}:=\int_{\mathcal O}^\oplus \pi_t m_{\mathcal O}(t)$.
Then  the induced representation $\widetilde{\pi_{\mathcal O}}:=\Ind_{\Ga_{\rm fc}}^\Ga \pi_{\mathcal O}$ is factorial for $\dot{\nu}$-almost every ${\mathcal O}$ and we have a direct integral decomposition  of the von Neumann algebra $L(\Ga)$ into factors
 $$
 L(\Ga)=\int_{{\Char(\GaFC)}/K_\Ga}^\oplus \widetilde{\pi_{\mathcal O}}(\Ga)'' d\dot{\nu}({\mathcal O}).$$
 \item[(ii)]  The Plancherel measure of  $\Ga$ 
 is the image $\Phi_*(\nu)$ of  $\nu$ under the map 
 $$
 \Phi: \Char(\Ga_{\rm fc})\to \Tr(\Ga), \qquad t\mapsto \int_{K_\Ga} \widetilde{t^g} dm(g),
 $$
 where  $m$ is the normalized Haar measure on $K_\Ga.$
\end{itemize}
\end{thmx} 
It is worth mentioning that the \emph{support} of $\mu$ was determined in \cite{Thoma-Plancherel}.
For an expression of the map $\Phi$ as in Theorem~\ref{Pro-PlancherelFC}.ii without reference to the 
group $K_\Ga$, see Remark~\ref{Rem-Pro-PlancherelFC}.

 As we next see, the Plancherel measure on $\GaFC$ can be explicitly described in the case of a linear group $\Ga.$
 We  first need to discuss the Plancherel formula  for a so-called \textbf{central group}, 
 that is, a central extension of a finite group.
 
 Let $\La$ be a central group. Then $\La$ is  of type I (see Theorem~\ref{Theorem-Thoma-Kaniuth}).
 In fact, $\widehat{\La}$ can be described as follows; let $r: \widehat{\La}\to \widehat{Z(\La)}$ be the 
 restriction map, where $Z(\La)$ is the center of $\La.$ Then, for $\chi\in  \widehat{Z(\La)},$ 
 every $\pi\in r^{-1}(\chi)$  is equivalent to a subrepresentation of the finite dimensional 
 representation $\Ind_{Z(\La)}^\La \chi$, by a generalized  Frobenius reciprocity
 theorem  (see \cite[Theorem 8.2]{Mackey0}); in particular, 
 $r^{-1}(\chi)$ is  finite. 
 The  Plancherel measure  $\nu$ on $\Char(\La)$ is, in principle,  easy to determine:
 we identify every  $\pi\in\widehat{\La}$ with its 
 normalized character given by $x\mapsto \dfrac{1}{\dim \pi}\Tr \pi(x)$;
for every Borel subset $A$ of $\Char(\La)$, we have
$$\nu (A)= \int_{\widehat{Z(\La)}} \dfrac{\# (A\cap r^{-1}(\chi)) }{\sum_{\pi \in r^{-1}(\chi)}(\dim \pi)^2} d\chi,$$
where $d\chi$  is the normalized Haar measure on the abelian compact group $\widehat{Z(\La)}$.

 \begin{corx}
 \label{Theo-Plancherel-Linear}
 (\textbf{The Plancherel measure for linear groups})
Let $\Ga$ be a countable linear group. 
\begin{itemize}
\item[(i)] $\Ga_{\rm fc}$ is a central group;
\item[(ii)] $K_\Ga$ coincides with $\Ad(\Ga)|_{\Ga_{\rm fc}}$ and is a finite group;
\item[(iii)] the Plancherel measure  of $\Ga$ is the image of  the Plancherel measure of $\Ga_{\rm fc}$ under the map 
 $\Phi: \Char(\Ga_{\rm fc}) \to \Tr(\Ga)$ given by 
 $$\Phi(t)=\dfrac{1}{\# \Ad(\Ga)|_{\Ga_{\rm fc}}}\sum_{s\in\Ad(\Ga)|_{\Ga_{\rm fc}}} t^s.
 $$
 \end{itemize}

 \end{corx}
 
 When the Zariski closure of the linear  group $\Ga$ is connected, the Plancherel measure of $\Ga$ has a particularly simple form.
 
  \begin{corx}
 \label{Cor-Theo-Plancherel-Linear}
 (\textbf{The Plancherel measure for linear groups-bis})
 Let $\mathbf G$ be a \emph{connected} linear algebraic group  over a field $\kk$ and let $\Ga$ be a countable Zariski dense subgroup of $\mathbf G$. 
 The Plancherel measure of $\Ga$ is the image of  the normalized
 Haar measure $d\chi$ on $\widehat{Z(\Ga)}$ under the map 
 $$ \widehat{Z(\Ga)}\to \Tr(\Ga),\qquad \chi \mapsto \widetilde{\chi}$$ 
and the Plancherel formula is given for every $ f\in \CCC[\Ga]$ by 
$$
\Vert f\Vert_2^2 =\int_{ \widehat{Z(\Ga)}}\mathcal{F }( (f^*\ast f)|_{Z(\Ga)})(\chi)d\chi,$$

where $\mathcal{F }$ is the Fourier transform on the abelian group $Z(\Ga).$
 
The previous conclusion holds in the following two cases:
 \begin{itemize}
 \item[(i)] $\kk$ is a countable field of characteristic $0$ and $\Ga= \mathbf G(\kk)$ is the group of $\kk$-rational points in $\mathbf G$;
 \item[(ii)] $\kk$ is a local field (that is, a non discrete locally compact field),
 $\mathbf{G}$ has no proper   $\kk$-subgroup $\mathbf{H}$ such that $(\mathbf G/\mathbf H)(\kk)$ is compact, and $\Ga$ is a lattice  in $\mathbf G(\kk).$
 \end{itemize}

 \end{corx}
Corollary~\ref{Cor-Theo-Plancherel-Linear} generalizes the Plancherel theorem obtained in \cite[Theorem 4]{C-PJ}  for $\Ga= \mathbf G(\QQ)$ and in \cite[Theorem 3.6]{PJ} for   $\Ga= \mathbf G(\ZZ)$, in the case where $\mathbf G$ is a unipotent linear algebraic group over $\QQ;$ indeed, $\mathbf G$ is connected  (since the 
 exponential map identifies $\mathbf G$ with its Lie algebra, as affine varieties) and these two results follow from  (i) and (ii) respectively.

In an appendix to this article, we use Theorem~\ref{Pro-Plancherel} to give a unified proof of Thoma's and Kaniuth's results (\cite{Thoma1},\cite{Thoma2}, \cite{Kaniuth}) as stated in the following theorem. 
For a group $\Ga,$ we denote by $[\Ga, \Ga]$ its commutator subgroup.
Recall  that $\Ga$ is said to be virtually abelian if it contains an abelian subgroup of finite index.

The regular representation $\la_\Ga$ is of type I (or type II)
if the von Neumann algebra $L(\Ga)$ is of type I (or type II); equivalently (see Corollaire 2 in \cite[Chap.~II, \S~3, 5]{Dixm--vN}),
if $\pi_t(\Ga)''$ is a finite dimensional factor (or a factor of type II)
for $\mu$-almost every $t\in \Char(\Ga)$
in the Plancherel decomposition $\la_\Ga= \int^\oplus_{\Char(\Ga)} \pi_t d\mu(t)$ 
from Theorem~\ref{Pro-Plancherel}.

\begin{thmx}
[\textbf{Thoma}, \textbf{Kaniuth}]
\label{Theorem-Thoma-Kaniuth}
Let $\Gamma$ be a  countable group.
The following properties are equivalent:
\begin{itemize}
\item[(i)] 
$\Gamma$ is type I;
\item[(ii)] 
$\Gamma$ is virtually abelian;
\item[(iii)] the regular representation $\la_\Ga$ is of type I;
\item[(iv)] every irreducible unitary representation of $\Ga$ is finite dimensional;
\item[(iv')] there exists an integer $n\geq 1$ such that every irreducible unitary representation of $\Ga$ 
has dimension $\leq n.$
\end{itemize}
Moreover, the following properties are equivalent:
\begin{itemize}
 \item[(v)] $\la_\Ga$ is of   type II;
 \item[(vi)] either $[\Gamma \,\colon \Ga_{\rm fc}] = \infty$ 
or $[\Gamma \,\colon \Ga_{\rm fc}] < \infty$ and $[\Ga, \Ga]$ is infinite.
\end{itemize}
\end{thmx}
Our proof of Theorem~\ref{Theorem-Thoma-Kaniuth}
 is not completely new as it uses several crucial ideas from \cite{Thoma1} and especially from \cite{Kaniuth}
 (compare with the remarks on p.336  after Lemma in \cite{Kaniuth}); however, we felt it could be useful
 to have a  short and common treatment of both results in the literature.
 Observe that the equivalence between (i)  and (iii) above does not carry over to non discrete groups (see \cite{Mackey2}).

\begin{remark}
\label{Rem-Theorem-Thoma-Kaniuth}
  Theorem~\ref{Theorem-Thoma-Kaniuth} holds also for   non countable discrete groups.
Write such a group $\Ga$ as  $\Ga= \cup_{j} H_j$ for  a directed net  of countable subgroups $H_j.$
 If  $L(\Ga)''$ is not of type II (or is of type I), then $L(H_j)$ is not of type $II$ (or is of type I)
 for $j$ large enough. This is the crucial tool for the   extension of the proof of
 Theorem~\ref{Theorem-Thoma-Kaniuth} to  $\Ga;$ for more details, proofs of Satz 1 and Satz 2 in  \cite{Kaniuth}.

\end{remark}

This paper is organized as follows. In Section~\ref{CenterVNA}, we recall the well-known description
of the center of the von Neumann algebra of a discrete group. 
Sections~\ref{ProofPL} and \ref{SS:PlancherelFC} contains  the proofs of
Theorems~ \ref{Pro-Plancherel} and \ref{Pro-PlancherelFC}.
In Section~\ref{SS:Plancherel-Linear}, we prove Corollaries~\ref{Theo-Plancherel-Linear} and \ref{Cor-Theo-Plancherel-Linear}. Section~\ref{Examples} is devoted to the explicit computation of the Plancherel formula
for a few examples of countable groups. Appendix~\ref{S: ThomaKaniuth} contains the proof
of Theorem~\ref{Theorem-Thoma-Kaniuth}.

\noindent
{\bf Acknowledgement.}\ We are grateful to Pierre-Emmanuel Caprace and Karl-Hermann Neeb for  useful  comments
on the first version of this paper.

 \section[Center of group von Neumann algebra]{On the center of the group von Neumann algebra}
\label{CenterVNA}

Let $\Ga$ be a countable group. We will often use the following well-known description of the center
 $\mathcal Z=\la(\Ga)''\cap \la(\Ga)'$  of $L(\Ga)= \la_\Ga(\Ga)''$. 
 
 Observe that $\la_\Ga(H)''$ is a von Neumann subalgebra of $L(\Ga)$, for every subgroup $H$ of $\Ga.$ 
For $h\in \Ga_{\rm fc},$ we set
$$T_{[h]}:= \la_\Ga(\mathbf{1}_{[h]})=\sum_{x\in [h]} \la_\Ga(x) \in \la_\Ga(\Ga_{\rm fc})'',$$ where $[h]$ denotes the $\Ga$-conjugacy class of $h.$

\begin{lemma}
\label{Lem-Center}
The center $\mathcal Z$ of $L(\Ga)= \la_\Ga(\Ga)''$ coincides with the closure of  the linear span of 
$\{T_{[h]}\mid h\in \Ga_{\rm fc}\}$, for the strong operator  topology; in particular, 
$\mathcal Z$ is contained in $ \la_\Ga(\Ga_{\rm fc})''.$
\end{lemma}
\begin{proof}

It is clear that  $T_{[h]} \in \mathcal Z$ for every $h\in \Ga_{\rm fc}.$
Observe also that the linear span of $\{T_{[h]}\mid h\in \Ga_{\rm fc}\}$
is a unital selfadjoint algebra; indeed, $T_{[h^{-1}]}= T_{[h]}^*$ for every 
$h\in  \Ga_{\rm fc}$ and $\{\mathbf{1}_{[h]}\mid h\in \Ga_{\rm fc}\}$ is a 
vector space basis of the algebra $\CCC[ \Ga_{\rm fc}]^\Ga$ of $\Ga$-invariant
functions in $\CCC[\GaFC]$.

Let $T\in \mathcal Z$. We have to show that $T\in \{T_{[h]}\mid h\in \Ga_{\rm fc}\}''.$
For every $\ga \in \Ga,$ we have 
$$
\begin{aligned}
\lambda_\Gamma(\gamma) \rho_\Gamma(\gamma) (T\delta_e)
\, = \, ( \lambda_\Gamma(\gamma) \rho_\Gamma(\gamma)T) \delta_e
\, = \, (T \lambda_\Gamma(\gamma) \rho_\Gamma(\gamma)) \delta_e
\, = \, T \delta_e.
\end{aligned}
$$
and this shows that  $f:=T\delta_e,$ which is a function in $\ell^2(\Gamma),$ is 
invariant under conjugation by $\ga.$
The support of  $f$ is therefore contained in $\Ga_{\rm fc}.$

Write $f= \sum_{[h]\in \mathcal C} c_{[h]}\mathbf{1}_{[h]}$ for a sequence
$(c_{[h]})_{[h]\in \mathcal C}$ of complex numbers  with $\sum_{[h]\in \mathcal C} \# [h]|c_{[h]}|^2 <\infty,$
where $\mathcal C$ is a set of representatives for the $\Ga$-conjugacy classes in  $\Ga_{\rm fc}.$
Let $\rho_\Ga$ be the right regular representation of $\Gamma$.
Since $T\in  \la_\Ga(\Ga)''$ and $\rho_\Ga(\Ga)\subset \la_\Ga(\Ga)'$,
we have, for every $x\in \Ga$,
$$T(\delta_x)= T\rho_\Ga(x) (\delta_e)= \rho_\Ga(x) (f)= 
\rho_\Ga(x) \left( \sum_{[h]\in \mathcal C} c_{[h]}\mathbf{1}_{[h]}\right)= 
\sum_{[h]\in \mathcal C} c_{[h]} T_{[h]}(\delta_x),$$ where the last sum is convergent in $\ell^2(\Ga).$
We also have $T^*(\delta_x)=\sum_{[h]\in \mathcal C}  \overline{c_{[h]}} T_{[h^{-1}]}(\delta_x).$

 Let $S\in \{T_{[h]}\mid h\in \Ga_{\rm fc}\}'.$ For every $x, y\in \Ga,$ 
 we have
$$
\begin{aligned}
\langle ST(\delta_x)\mid \delta_y\rangle&=\langle S\left(\sum_{[h]\in \mathcal C} c_{[h]} T_{[h]}(\delta_x)\right)\mid \delta_y\rangle=\langle\sum_{[h]\in \mathcal C} c_{[h]} ST_{[h]}(\delta_x)\mid \delta_y\rangle\\
&= \sum_{[h]\in \mathcal C}c_{[h]} \langle T_{[h]}S(\delta_x)\mid \delta_y\rangle=
\langle S(\delta_x)\mid \sum_{[h]\in \mathcal C} \overline{c_{[h]}}T_{[h^{-1}]}(\delta_y)\rangle\\
&=\langle S(\delta_x)\mid T^*(\delta_y)\rangle=\langle TS(\delta_x)\mid \delta_y)\rangle\\
  \end{aligned}
 $$
and it follows that $ST=TS.$ 
\end{proof}

The following well-known corollary shows that  the Plancherel measure is 
the Dirac measure at $\delta_e$ in the case where $\Ga$ is ICC group, that is, when $\GaFC=\{e\}.$
\begin{corollary}
\label{Cor-Lem-Center}
 Assume that $\Ga$ is ICC. Then 
$L(\Ga)=\la_\Ga(\Ga)''$ is a factor.
\end{corollary}

\section{Proof of Theorem~\ref{Pro-Plancherel}}
\label{ProofPL}

Consider  a direct integral decomposition 
$ \int_X^\oplus \pi_x d\mu(x)$
of $\la_\Ga$ associated to the centre $\mathcal Z$  of $L(\Ga)= \la_\Ga(\Ga)''$
(see [8.3.2]\cite{Dixm--C*});  so,  $X$ is a standard Borel space 
equipped with a probability measure $\mu$
and  $(\pi_x, \Hi_x)_{x \in X}$ is measurable field of   representations of $\Ga$
over $X$, such that there exists an isomorphism of Hilbert spaces
$$U \, \colon \ell^2(\Gamma) \to \int_X^\oplus \Hi_x d\mu(x)$$
which transforms $\la_\Ga$ into $ \int^\oplus_{X} \pi_xd\mu(x)$
and for which $U \mathcal Z U^{-1}$ is
the algebra of diagonal operators on $\int_X^\oplus \Hi_x d\mu(x)$.
(Recall that a diagonal operator on $\int_{X}^\oplus \H_x d \mu(x)$ 
is an operator of the form $\int_{X}^\oplus \vfi(x) I_{\H_x} d \mu(x)$
for an essentially bounded measurable function $\vfi: X\to \CCC$.)

Then, upon disregarding a subset of $X$ of $\mu$-measure 0,  the  following  holds (see [8.4.1]\cite{Dixm--C*}):
\begin{itemize}
\item[(1)] $\pi_x$ is a factor representation for every $x\in X;$
\item[(2)] $\pi_x$ and $\pi_y$ are disjoint
for every $x, y \in X$ with $x \ne y$;
\item[(3)] we have $U \la_\Ga(\Ga)'' U^{-1} = \int^\oplus_{X} \pi_x(\Ga)'' d\mu(x).$
\end{itemize}

Let $\rho_\Ga$ be the right regular representation of $\Gamma$.
Let $\ga\in \Gamma.$ 
Then $U \rho_\Ga(\gamma) U^{-1}$ commutes with every  diagonalisable operator
on $\int_X^\oplus \Hi_x d\mu(x)$,  since $\rho_\Ga(\ga)\in L(\Ga)'$.
 It follows (see \cite[Chap. II, \S 2, No 5, Th\'eor\`eme 1]{Dixm--vN}) that $U \rho_\Ga(\gamma) U^{-1}$ 
is a decomposable operator, that is, there exists a 
measurable field of unitary operators $x \mapsto \sigma_x(\ga)$ such that 
$U \rho_\Gamma(\gamma) U^{-1}= \int^\oplus_{X} \sigma_x(\gamma) d\mu(x).$
So, we have  a measurable field $x \mapsto \sigma_x$ of representations of $\Gamma$
in $\int_X^\oplus \Hi_x d\mu(x)$ such that
$$
U \rho_\Ga(\gamma) U^{-1} \, = \, \int^\oplus_{X} \sigma_x(\gamma) d\mu(x)
\hskip.5cm \text{for all} \hskip.2cm
\gamma \in \Gamma.
$$

Let $(\xi_x)_{x \in X} \in \int_X^\oplus \Hi_x d\mu(x)$ be the image of 
 $\delta_e\in \ell^2(\Ga)$
under $U.$
We claim that $\xi_x$ is a cyclic vector for $\pi_x$ and $\sigma_x,$
for $\mu$-almost every $x \in X$.
Indeed,  since $\delta_e\in \ell^2(\Ga)$ is a cyclic vector for both $\lambda_\Gamma$
and $\rho_\Gamma,$ 
$$\{(\pi_x(\gamma)\xi_x)_{x\in X}\mid \gamma \in \Gamma\} \quad \text{and} \quad \{(\sigma_x(\gamma)\xi_x)_{x\in X}\mid \gamma \in \Gamma\}$$ 
are countable total subsets of $\int_X^\oplus \Hi_x d\mu(x)$
and the claim follows from a general fact about direct integral of Hilbert spaces
(see Proposition 8 in Chap. II, \S 1 of \cite{Dixm--vN}).

Since $\lambda_\Gamma(\gamma) \delta_e = \rho_\Gamma(\gamma^{-1}) \delta_e$
for every $\gamma \in \Gamma$ and since $\Gamma$ is countable, upon neglecting a subset of $X$ of $\mu$-measure 0, 
we can assume that 
\begin{itemize}
\item[(4)]
$\pi_x(\gamma) \xi_x = \sigma_x(\gamma^{-1}) \xi_x$;
\item[(5)]
$\pi_x(\gamma) \sigma_x(\gamma') = \sigma_x(\gamma') \pi_x(\gamma)$;
\item[(6)]
$\xi_x$ is a cyclic vector for  both $\pi_x$ and $\sigma_x$,
\end{itemize}
for all $x \in X$ and all $\gamma, \gamma' \in \Gamma$. 

Let  $x \in X$ and  let $\varphi_x$ be the function of positive type on $\Gamma$ defined by 
$$
\varphi_x(\gamma) \, = \, \langle \pi_x(\gamma) \xi_x \mid \xi_x \rangle
\hskip.5cm \text{for every } \hskip.2cm
\gamma \in \Gamma.
$$
We claim that $\vfi_x\in \Char(\Ga).$ Indeed, using (4) and (5), we have, for every $\gamma_1, \gamma_2 \in \Gamma$,
$$
\begin{aligned}
\varphi_x(\gamma_2 \gamma_1 \gamma_2^{-1})
&= 
\langle \pi_x(\gamma_2 \gamma_1 \gamma_2^{-1}) \xi_x \mid \xi_x \rangle
= 
\langle \pi_x(\gamma_2 \gamma_1) \sigma_x(\gamma_2) \xi_x \mid \xi_x \rangle
\\
&= 
\langle \sigma_x(\gamma_2) \pi_x(\gamma_2\gamma_1) \xi_x \mid \xi_x \rangle
=
\langle \pi_x( \gamma_1) \xi_x \mid \pi_x (\gamma_2^{-1}) \sigma_x(\gamma_2^{-1}) \xi_x \rangle
\\
&=
\langle \pi_x( \gamma_1) \xi_x \mid \xi_x \rangle
=
\varphi_x(\gamma_1).
\end{aligned}
$$
So, $\vfi_x$ is conjugation invariant and hence 
$\vfi_x\in \Tr(\Ga).$  Moreover, $\vfi_x$ is an extreme point in $\Tr(\Ga),$ 
since $\pi_x$ is factorial and $\xi_x$ is a cyclic vector for $\pi_x.$

Finally, since $U \colon \ell^2(\Gamma) \to \int_X^\oplus \Hi_x d\mu(x)$ is an isometry, we have for every 
$f\in \CCC[\Ga],$
$$
\begin{aligned}
\Vert f\Vert^2&= f^*\ast f(e)= \langle \la_\Ga(f^*\ast f) \delta_e\mid \delta_e \rangle\\
&=\Vert \la_\Ga(f) \delta_e\Vert^2= \Vert U(\la_\Ga(f) \delta_e)\Vert^2\\
&=\int_X\Vert \pi_x(f) \xi_x\Vert^2 d\mu(x)= \int_X\vfi_x(f^*\ast f) d\mu(x).
\end{aligned}
$$

The measurable map $\Phi:X\to \Char(\Ga)$ given by 
$\Phi(x)= \varphi_x$ is  injective, since 
$\pi_x$ and $\pi_y$ are disjoint by (2) and hence $\varphi_x\neq \varphi_y$
for $x, y \in X$ with $x \ne y$.
It follows that $\Phi(X)$ is a Borel subset of $\Char(\Ga)$ and that
$\Phi$ is a Borel isomorphism between $X$ and $\Phi(X)$ 
(see \cite[Theorem 3.2]{Mackey-Borel}).
Pushing forward  $\mu$ to $\Char(\Ga)$ by $\Phi$,
we  can therefore assume without 
loss of generality that $X=\Char(\Ga)$ and that $\mu$
is a probability measure on $\Char(\Ga)$. 
With this identification,  it is clear that  Items (i), (ii) and (iii) of Theorem~\ref{Pro-Plancherel}
are satisfied and that the Plancherel formula holds.

It remains to show the uniqueness of $\mu.$
Let $\nu$ any probability measure on $\Char(\Ga)$ such that 
the Plancherel formula.
By polarization, we have then
$\delta_e= \int_{\Char(\Ga)}t d\nu(t)$, which is an integral decomposition 
of $\delta_e\in \Tr(\Ga)$ over extreme points of the  convex set $\Tr(\Ga).$
The uniqueness of such a decomposition implies that $\nu=\mu.$

\begin{remark}
\label{SupportPlancherel}
(i) For $\mu$-almost every $t\in \Char(\Ga),$ we have $t=0$ on $ \Ga\setminus \Ga_{\rm fc}.$ 
Indeed, let $\gamma \notin \GaFC$.  Then  
$\langle \la_\Ga(\ga) \la_\Ga(h)\delta_e\mid \delta_e\rangle=0$
for every $h\in \Ga_{\rm fc}$ and hence 
$$
\leqno{(*)}\qquad \langle \la_\Ga(\ga) T\delta_e\mid \delta_e\rangle=0 \tout T\in \la_\Ga(\Ga_{\rm fc})''.
$$
With the notation as in the proof above, let $E$ be a  Borel subset of $X.$
Then $T_E:=U^{-1}P_E U$ is a projection  in $\mathcal Z,$  where $P_E$ is the diagonal  operator 
$\int_X^\oplus \mathbf{1}_E(x) I_{\H_x} d\mu(x)$. It follows from Lemma~\ref{Lem-Center} and $(*)$
that 
$$
\int_E \vfi_x(\ga) d\mu(x)=\langle T_E\la_\Ga(\ga)\delta_e\mid \delta_e\rangle=\langle \la_\Ga(\ga)T_E\delta_e\mid \delta_e\rangle=0.
$$
Since this holds for every Borel subset $E$ of $X,$
this implies that $\varphi_x(\gamma) = 0$ for $\mu$-almost every $x \in X$.

As $\Gamma$ is countable, for $\mu$-almost every $x \in X$, we have $\varphi_x(\gamma) = 0$ for every $\gamma \notin \GaFC$.
\par

\noindent
(ii) Let $\la_\Ga=\int_{\Char(\Ga)}^\oplus \pi_t d\mu(t)$, $\rho_\Ga=\int_{\Char(\Ga)}^\oplus \sigma_t d\mu(t)$, and 
$\delta_e= (\xi_t )_{t\in \Char(\Ga)}$  be the   decompositions as above.
For $\mu$-almost every $t\in\Char(\Ga),$  the linear map  $$\pi_t(\Ga)''\to  \H_t, \,   T\mapsto  T\xi_t$$
is injective. Indeed, this follows from the fact that 
$\xi_t$ is cyclic for $\sigma_t$ and that $\sigma_t(\Ga)\subset \pi_t(\Ga)'.$

\end{remark}

\section{Proof of Theorem~\ref{Pro-PlancherelFC}}
\label{SS:PlancherelFC}
Set  $N:= \Ga_{\rm fc}$ and $X:=\Char(N).$
Consider  the direct integral decomposition 
$\int_X^\oplus \pi_t d\nu(t)$ of $\la_{N}$
into factor representations $(\pi_t, \K_t)$ of  $N$
with corresponding traces $t\in X,$ as in Theorem~\ref{Pro-Plancherel}.

Let $K_\Ga$ be the compact group which is the closure 
in $\Aut(N)$ of $\Ad(\Ga)|_{N}$.
Since the quotient space 
$X/K_\Ga$ is a standard Borel space,
there  exists  a Borel section $s:X/K_\Ga \to X$
for the projection map $X\to X/K_\Ga.$
Set $\Omega:= s(X/K_\Ga)$.Then $\Omega$ is a Borel transversal for 
$X/K_\Ga$.
The Plancherel measure $\nu$ can  accordingly be decomposed over $\Omega:$  we have
$$
\nu(f)=\int_{\Omega} \int_{\mathcal{O}_\omega}f(t)dm_{\omega}(t)d\dot{\nu}(\omega)
$$
for every bounded measurable function $f$ on $\Char(\GaFC),$ 
where $m_{\omega}$ be the unique normalized $K_\Ga$-invariant probability measure
 on the $K_\Ga$-orbit $\mathcal O_\omega$ of $\omega$ and $\dot{\nu}$ is the image of $\nu$ under $s.$
 
 Let $\omega\in \Omega$ and set 
 $$ \pi_{\mathcal O_\omega}:=\int_{\mathcal O_\omega}^\oplus \pi_t dm_{\omega}(t),$$
which is a unitary representation of $N$ on the Hilbert space  
$$\K_\omega:=\int_{\mathcal O_\omega}^\oplus \K_t dm_{\omega}(t).$$

For $g\in \Aut(N),$ let  $\pi_{\mathcal O_\omega}^g$ be the conjugate representation of $N$ on $\K_\omega$
given by $\pi_{\mathcal O_\omega}^g(h)= \pi_{\mathcal O_\omega}(g(h))$ for $h\in N.$

 \vskip.5cm
 \textbf{Step  1}   There exists a unitary representation $U_\omega: g\mapsto U_{\omega,g}$ of $K_\Ga$
 on $\K_\omega$ 
  such that 
 $$U_{\omega,g} \pi_{\mathcal O_\omega}(h) U_{\omega,g}^{-1}= \pi_{\mathcal O_\omega}(g(h))
 \qquad\text{for all} \quad g\in K_\Ga, h\in N;$$
 in particular,
 $\pi_{\mathcal O_\omega}^g$ is equivalent to  $\pi_{\mathcal O_\omega}$ for every $g\in K_\Ga.$

Indeed, observe that the representations $\pi_t$ for $t\in \mathcal O_\omega$ are  conjugate  to each other (up to equivalence) and  may therefore be considered as defined on the same Hilbert space.

Let $g\in K_\Ga.$ Then  $\pi_{\mathcal O_\omega}^g$  is equivalent to 
$\int_{\mathcal O_\omega}^\oplus {\pi_t}^g dm_{\omega}(t).$
Define a linear operator $U_{\omega,g}: \K_\omega\to\K_\omega$ by 
$$
U_{\omega,g}\left((\xi_t)_{t\in \mathcal O_\omega}\right)= (\xi_{t^g})_{t\in \mathcal O_\omega}
\tout (\xi_t)_{t\in \mathcal O_\omega}\in \K_\omega.
$$
Then $U_{\omega,g}$ is an isometry, by $K_\Ga$-invariance of the measure $m_\omega.$
It is readily checked that $U_{\omega}$ intertwines $\pi_{\mathcal O_\omega}$ and $\pi_{\mathcal O_\omega}^g$
and that $U_{\omega}$ is a homomorphism.
To show that $U_{\omega}$ is a representation of $K_\Ga,$
it remains to prove that  $g\mapsto U_{\omega,g}\xi$ is continuous for every  $\xi\in \K_\omega.$

For this, observe that $\K_\omega$ can be identified with  the Hilbert space $L^2(\mathcal O_\omega, m_\omega) \otimes \K$,
where $\K$ is the common Hilbert space of the $\pi_t$'s for $t\in \mathcal O_\omega;$
under this identification, $U_{\omega}$ 
corresponds to $\kappa\otimes I_{\K}$, where $\kappa$ is the Koopman representation of $K_\Ga$ on $L^2(\mathcal O_\omega, m_\omega)$ associated to the action $K_\Ga\curvearrowright \mathcal O_\omega$ (for the fact that 
$\kappa$ is indeed a representation of $K_\Ga,$  see \cite[A.6]{BHV}) and the claim follows.
 
\vskip.7cm
 Next, let 
 $$\widetilde{\pi_{\mathcal O_\omega}}:=\Ind_{N}^\Ga \pi_{\mathcal O_\omega}$$
 be the  representation of $\Ga$ induced by $\pi_{\mathcal O_\omega}.$ 
 
 We recall how 
  $\widetilde{\pi_{\mathcal O_\omega}}$ can be realized 
on $\ell^2(R, \K_\omega)= \ell^2(R) \otimes  \K_\omega $, where $R\subset \Ga$ is a set of representatives  for the cosets
of $N$ with $e\in R.$  For every $\ga \in \Ga$ and $r \in R$, 
let $c(r, \ga)\in N$ and $r\cdot \ga\in R$ be such that
$r\ga= c(r, \ga)r\cdot \ga.$
Then $\widetilde{\pi_{\mathcal O_\omega}}$ is given  on $\ell^2(R, \Ki_\omega)$ by
$$
(\widetilde{\pi_{\mathcal O_\omega}}(\ga) F) (r) \, = \, \pi_{\mathcal O_\omega} (c(r, \ga)) (F(r\cdot \ga)) 
\tout F\in \ell^2(R, \Ki_\omega).
$$
\vskip.5cm
 \textbf{Step  2} We claim that
 there exists  a unitary map 
$$\widetilde{U_{\omega}}: \ell^2(R, \K_\omega) \to \ell^2(R, \K_\omega)$$
which  intertwines   the representation $I_{ \ell^2(R)} \otimes \pi_{\mathcal O_\omega}$ and the restriction 
 $\widetilde{\pi_{\mathcal O_\omega}}|_N$ of $\widetilde{\pi_{\mathcal O_\omega}}$ to $N$;
 moreover, $\omega \to \widetilde{U_{\omega}}$ is a measurable field of unitary operators on $\Omega.$

Indeed, we have an orthogonal 
decomposition 
$$\ell^2(R, \Ki_\omega)= \oplus_{r\in R}(\delta_r\otimes \K_\omega)$$
into $\widetilde{\pi_{\mathcal O_\omega}}(N)$-invariant; moreover, the action of $N$ 
on every copy $\delta_r\otimes \K_\omega$ is given by  $ \pi_{\mathcal O_\omega}^r$. 
For every $r\in R,$ the unitary operator
$U_{\omega,r }: \K_\omega\to \K_\omega$ from Step 1 intertwines $\pi_{\mathcal O_\omega}$ and $\pi_{\mathcal O_\omega}^r.$
In view of the explicit formula of $U_{\omega,r },$ the  field $\omega \to U_{\omega,r}$ is measurable on $\Omega.$

Define a unitary operator $\widetilde{U_{\omega}}: \ell^2(R, \Ki_\omega) \to \ell^2(R, \Ki_\omega)$ by 
$$\widetilde{U_{\omega}} (\delta_r \otimes \xi)= \delta_r\otimes U_{\omega,r }(\xi) \tout \xi \in \K_\omega.$$
Then $\widetilde{U_{\omega}}$ intertwines $I_{ \ell^2(R)} \otimes \pi_{\mathcal O_\omega}$ and  $\widetilde{\pi_{\mathcal O_\omega}}|_N$; moreover, $\omega \to \widetilde{U_{\omega}}$ is a measurable field on $\Omega.$

\vskip.7cm

Observe that the representation $\la_N$ is equivalent to  $\int_{\Omega}^\oplus  \pi_{\mathcal O_\omega} d\dot{\nu}(\omega).$
 Since $\la_\Ga$ is equivalent  to $\Ind_N^\Ga \la_N,$ it follows that $\la_\Ga$ is equivalent to $\int_{\Omega}^\oplus \widetilde{\pi_{\mathcal O_\omega}} d\dot{\nu}(\omega).$

 In the sequel, we will identify  the representations  $\la_N$ on $\ell^2(N)$ and $\la_\Ga$
 on $\ell^2(\Ga)$ with respectively the representations 
$$
\int_{\Omega}^\oplus \pi_{{\mathcal O}_\omega}d \dot{\nu}(\omega)
\quad \text{on}\quad
 \K:=\int_{\Omega}^\oplus  \K_\omega d\dot{\nu}(\omega)$$
and 
 $$
 \int_{\Omega}^\oplus \widetilde{\pi_{\mathcal O_\omega}} d\dot{\nu}(\omega) 
\quad \text{on}\quad
 \H:=\int_{\Omega}^\oplus  \ell^2(R, \K_\omega) d\dot{\nu}(\omega).
 $$

 \vskip.5cm
 \textbf{Step 3} The representations $\widetilde{\pi_{\mathcal O_\omega}}$ are factorial 
and are mutually disjoint, outside a subset of $\Omega$ of $\dot{\nu}$-measure 0.

To show this, it suffices to prove (see \cite[8.4.1]{Dixm--C*}) that  the algebra 
$\D$ of diagonal operators in $\L(\H)$ coincides with the center $\Z$ of $\la_\Ga(\Ga)''$. 

Let us first prove that $D\subset \Z.$ For this, we only have to prove
 that  $\D \subset \la_\Ga(\Ga)'',$ since it is clear that $\D \subset \la_\Ga(\Ga)'.$

 By Step 2, for every $\omega\in \Omega,$  there exists a measurable field $\omega \to \widetilde{U_{\omega}}$
 of unitary operators
 $\widetilde{U_{\omega}}: \ell^2(R, \K_\omega) \to \ell^2(R, \K_\omega)$ intertwining  $I_{ \ell^2(R)} \otimes \pi_{\mathcal O_\omega}$
 and $\pi_{\mathcal O_\omega}|_N$.
So, 
$$\widetilde{U}:= \int_{\Omega}^\oplus  \widetilde{U_{\omega}} d\dot{\nu}(\omega)$$
 is a unitary operator on $\H$ 
which intertwines $I_{ \ell^2(R)} \otimes \la_N$ and $\la_\Ga|_N$;
it is obvious that $\widetilde{U}$ commutes with the diagonal operators on $\H.$

Let $\vfi: \Omega\to \CCC$ be a measurable essential bounded function on $\Omega$. 
By Theorem~\ref{Pro-Plancherel}.ii, the   corresponding diagonal operator 
$$T=\int_{\Omega}^\oplus \vfi(\omega) I_{\K_\omega} d\dot{\nu}(\omega)$$
on $\K$ belongs to   $\la_N(N)''.$ 
For the  corresponding diagonal operator  
$$\widetilde T=\int_{\Omega}^\oplus \vfi(\omega) I_{\ell^2(R, \K_\omega)} \dot{\nu}(\omega)$$ on 
 $\H$, we have $\widetilde{T}= I_{ \ell^2(R)} \otimes T$.
 So, $ \widetilde{T}$ belongs to $(I_{ \ell^2(R)} \otimes \la_N)(N)''.$
 Since, $\widetilde{U}$ commutes with $\widetilde{T}$ and intertwines   $I_{ \ell^2(R)} \otimes \la_N$ and $\la_\Ga|_N$,
 it follows that 
 $$\widetilde{T}= \widetilde{U}( I_{ \ell^2(R)} \otimes T)\widetilde{U}^{-1}\in \la_\Ga(N)''\subset \la_\Ga(\Ga)''.$$
 
 So, we have shown that $\D \subset \Z.$ 
 Observe that this implies (see Th\'eor\`eme 1 in Chap. II, \S3 of \cite{Dixm--vN})
  that $L(\Ga)$ is  the direct integral $ \int_{\Omega}^\oplus \widetilde{\pi_{\mathcal O_\omega}}(\Ga)''d\dot{\nu}(\omega)$
  and  that
  $\Z$ is the direct integral $ \int_{\Omega}^\oplus \Z_\omega d\dot{\nu}(\omega) ,$
 where $\Z_\omega$ is the center of $\widetilde{\pi_{\mathcal O_\omega}}(\Ga)''.$

Let  $\widetilde{T}\in \Z.$ Then $\widetilde{T}\in \la_\Ga(N)'',$
by Lemma~\ref{Lem-Center}. So, 
$\widetilde{T}= \int_{\Omega}^\oplus T_\omega d\dot{\nu}(\omega),$
where $T_\omega$ belongs to the center of  $\widetilde{\pi_{\mathcal O_\omega}}(N)'',$ for $\dot{\nu}$-almost every $\omega.$ Since $\widetilde{\pi_{\mathcal O_\omega}}|_N$ is equivalent to $I_{ \ell^2(R)} \otimes \pi_{\mathcal O_\omega}$ 
and since $ \pi_{\mathcal O_\omega}$ and hence $I_{ \ell^2(R)} \otimes \pi_{\mathcal O_\omega}$
is a factor representation, it follows that $T_\omega$ is a scalar operator, for  $\dot{\nu}$-almost every $\omega.$
 So,  $\widetilde{T}\in \D.$

\vskip.7cm

As a result, we have a  decomposition 
$$\la_\Ga=\int_{\Omega}^\oplus \widetilde{\pi_{\mathcal O_\omega}} d\dot{\nu}(\omega)$$
of $\la_\Ga$ as a direct integral of  pairwise disjoint factor representations.
By the argument of the proof of Theorem~\ref{Pro-Plancherel},
it follows that, for  $\dot{\nu}$-almost every $\omega\in X,$
 there exists a cyclic unit vector $\xi_\omega\in\ell^2(R, \K_\omega)$ 
 for $\widetilde{\pi_{\mathcal O_\omega}}$ so that 
 $$\vfi_\omega:= \langle \widetilde{\pi_{\mathcal O_\omega}}(\cdot)\xi_\omega\mid \xi_\omega\rangle$$
  belongs to $\Char(\Ga)$. In particular, 
 $$\widetilde{\mathcal M_\omega}:=\widetilde{\pi_{\mathcal O_\omega}}(\Ga)''$$ 
 is a factor of type II$_1$ and its normalized trace 
 is the extension of $\vfi_\omega$ to    $\widetilde{\mathcal M_\omega}$, which we again denote by 
 $\vfi_\omega.$ Our next goal is to  determine $\vfi_\omega$ in terms of the character 
 $\omega\in \Char(N).$

 Fix $\omega \in \Omega$ such that $\widetilde{\mathcal M_\omega}$ is a factor. 
We identity the Hilbert space $\K_\omega$ of $ \pi_{\mathcal O_\omega}$ with the subspace
 $\delta_e\otimes \K_\omega$ and so $ \pi_{\mathcal O_\omega}$ with a subrepresentation of the restriction 
 of  $\widetilde{\pi_{\mathcal O_\omega}}$ to $N.$
 
 Let  $\eta_\omega\in \K_\omega$  be a cyclic vector for $\pi_{\mathcal O_\omega}$
 such  that 
 $$\omega=\langle\pi_{\mathcal O_\omega}(\cdot) \eta_\omega\mid \eta_\omega\rangle.$$
 For $g\in K_\Ga,$ define a normal state $\psi_{\omega,g}$ on  $\widetilde{\mathcal M_\omega}$
  by the formula
  $$
 \psi_{\omega,g}(\widetilde T)= \langle\widetilde T U_{g, \omega}^{-1} \eta_\omega\mid U_{g, \omega}^{-1}\eta_\omega\rangle \tout \widetilde T\in \widetilde{\mathcal M_\omega},
  $$
  where $U_{g, \omega}$ is the unitary operator on $\K_\omega$ from Step 1.
 
Consider the linear functional $\psi_\omega:\widetilde{\mathcal M_\omega}\to \CCC$  given by 
  $$\psi_\omega (\widetilde T)= 
   \int_{K_\Ga}   \psi_{\omega,g}(\widetilde T) dm(g) \tout \widetilde T\in 
  \widetilde{\mathcal M_\omega},$$
where $m$ is the normalized Haar measure on $K_\Ga.$

 \vskip.5cm
 \textbf{Step 4} We claim that ${\psi_\omega}$ is a normal state on 
 $\widetilde{\mathcal M_\omega}$
 
 Indeed, it is clear that ${\psi_\omega}$ is a state on  
 $\widetilde{\mathcal M_\omega}.$
  Let $(\widetilde T_n)_{n}$ be an increasing sequence
 of positive operators  in  $\widetilde{\mathcal M_\omega}$ with $\widetilde T= \sup_{n}\widetilde T_n \in \widetilde{\mathcal M_\omega}$.
 
 For every $g\in K_\Ga,$ the sequence $(\psi_{\omega,g}(\widetilde T_n))_n$ is increasing 
 and its  limit is $ \psi_{\omega,g}(\widetilde T)$. It follows from the monotone convergence theorem
 that 
 $$
 \lim_n \psi_\omega(T_n)= \lim_n\int_{K_\Ga}\psi_{\omega,g}(T_n) dm(g)= \int_{K_\Ga} \psi_{\omega,g}(T) dm(g)= \psi_\omega(T).
 $$
 So, $\psi_\omega$ is normal, as $\widetilde{\mathcal M_\omega}$ acts on a separable Hilbert space.

    \vskip.5cm
 \textbf{Step 5} We claim that, 
 writing $\ga$ instead of $\widetilde{\pi_{\mathcal O_\omega}}(\ga)$ for $\ga \in \Ga$, we have 
 $$
 \psi_\omega(\ga)=
 \begin{cases} \int_{K_\Ga} \omega^g(\ga) dm(g) \qquad \text{if}\qquad \ga\in N\\
 0\qquad \text{if}\qquad\ga\notin N.
 \end{cases}
 $$
Moreover,  $ \psi_\omega$ coincides with the trace $\vfi_\omega$ on 
 $\widetilde{\mathcal M_\omega}$ from above. 

Indeed, Let $\ga \in N.$ Since
 $$
  U_{\omega,g} \pi_{\mathcal O_\omega}(\ga) U_{\omega,g}^{-1}= \pi_{\mathcal O_\omega}(g(\ga)),
   $$
 we have 
 $$
 \begin{aligned}
 \psi_\omega(\ga)&=  \int_{K_\Ga} \langle\widetilde \pi_{\mathcal O_\omega}(g(\ga)) \eta_\omega\mid \eta_\omega\rangle dm(g)=  \int_{K_\Ga} \langle \pi_{\mathcal O_\omega}(g(\ga)) \eta_\omega\mid \eta_\omega\rangle dm(g)\\
 &= \int_{K_\Ga}\omega^g(\ga) dm(g).
  \end{aligned}
  $$
  Let $\ga \in \Ga\setminus N.$ Then, by the usual properties of an induced representation, 
   $\widetilde\pi_{\mathcal O_\omega}(\ga) (\K_\omega)$ is orthogonal to $\K_\omega$. It follows that
   $$
    \psi_{\omega,g}(\ga)= \langle \widetilde\pi_{\mathcal O_\omega}(\ga)U_{g, \omega}^{-1} \eta_\omega\mid U_{g, \omega}^{-1}\eta_\omega\rangle=0
    $$
    and hence $ \psi_\omega(\ga)= 0.$

 In particular, this shows  that $\psi_\omega$ is a $\Ga$-invariant state on $\widetilde{\mathcal M_\omega}$;
   since $\psi_\omega$ is normal (Step 4), it follows that $\psi_\omega$ is a trace on $\widetilde{\mathcal M_\omega}$.
 As $\widetilde{\mathcal M_\omega}$ is a factor (see Step 3), the fact that $\widetilde{ \psi_\omega}=\vfi_\omega$ follows from the uniqueness of normal traces on factors (see Corollaire p. 92 and Corollaire 2 p.83 in  \cite[Chap. I, \S 6]{Dixm--vN}).

    \vskip.5cm
 \textbf{Step 6} 
 The Plancherel measure $\mu$ on $\Ga$ is the image of $\nu$
 under the map $\Phi$ as in the statement of Theorem~\ref{Pro-PlancherelFC}.ii.

  Indeed, for $ f\in \CCC[\Ga]$, we have by Step 5
  $$
 \begin{aligned}
\Vert f\Vert_2^2 &=\int_{\Omega} \Vert \widetilde{\pi_{\mathcal O_\omega}}(f)\Vert^2 d\dot{\nu}(\omega)=\int_{\Omega} \psi_\omega\left( f^*\ast f\right) d\dot{\nu}(\omega)\\
 &=\int_{\Omega} \int_{K_\Ga}   \omega^g((f^*\ast f)|_N) dm(g) d\dot{\nu}(\omega)\\
&=\int_{\Omega} \int_{\mathcal O_\omega} \left (\int_{K_\Ga}   \omega^g((f^*\ast f)|_N) dm(g)\right)d{m_\omega}(t)d\dot{\nu}(\omega)\\
&=\int_{\Char(\GaFC)} \int_{K_\Ga}   t^g((f^*\ast f)|_N)dm(g)d\nu(t)\\
&= \int_{\Char(\Ga)} \Phi(t) (f^*\ast f) d\nu(t)
\end{aligned}
$$
and the claim follows.
    
  \begin{remark}
\label{Rem-Pro-PlancherelFC}
The map $\Phi$ in Theorem~\ref{Pro-PlancherelFC}.ii can be described without reference to the group $K_\Ga$ as follows. For $t\in \Char(\GaFC)$ and $\ga\in \Ga,$ we have
$$\Phi(t)(\ga)=
\begin{cases}
 \dfrac{1}{\#[\ga]}\sum_{x\in [\ga]} t(x)\qquad \text{if}\qquad \ga\in \Ga_{\rm fc}\\
 0 \qquad \text{if}\qquad  \ga\notin \Ga_{\rm fc}
 \end{cases},
 $$
 where $[\ga]$ denotes the $\Ga$-conjugacy class of $\ga.$
 Indeed, it suffices to consider the case where $\ga\in \GaFC$. The
 stabilizer $K_0$ of $\ga$ in $K_\Ga$ is an open and hence  cofinite subgroup
 of $K_\Ga;$ in particular, the $K_\Ga$-orbit of $\ga$ coincides
 with the $\Ad(\Ga)$-orbit of $\ga$ and so $ \{\Ad(x)\mid x\in [\ga]\}$ is a system
 of representatives for $K_\Ga/K_0.$
 Let $m_0$ be the normalized Haar measure on $K_0.$
 The normalized Haar measure $m$ on $K_\Ga$ is then given by  
 $m (f)=  \dfrac{1}{[\ga]}\sum_{x\in [\ga]} \int_{K_0}f(\Ad(x)g)dm_0(g)$ for every continuous function $f$ on $K_\Ga.$ It follows that
  $$
  \Phi(t)(\ga)=\int_{K_\Ga} t(g (\ga))dm(g)=
  \dfrac{1}{[\ga]}\sum_{x\in [\ga]} t(x).
  $$
  \end{remark}

\section{Proofs of Corollary~\ref{Theo-Plancherel-Linear} and Corollary~\ref{Cor-Theo-Plancherel-Linear}}\label{SS:Plancherel-Linear}
Let $\Ga$ be a countable linear group. So, $\Ga$ is a subgroup of $GL_n(\kk)$ for a field 
$\kk$, which may be assumed to be algebraically closed. Let $\mathbf G$ be the closure of $\Ga$ in the Zariski topology of $GL_n(\kk)$ and let $\mathbf G_0$ be the irreducible  component of $\mathbf G.$ As is well-known,  
$\mathbf G_0$  has finite index in $\mathbf G$  and hence $\Ga_0:= \mathbf G_0\cap \Ga$ is a normal subgroup of finite index in $\Ga$ 

Let $\ga\in \Ga_{\rm fc}$. On the one hand, the centralizer $\Ga_\ga$ of $\ga$ in $\Ga$  is a subgroup of finite index of $\Ga$; therefore,  the irreducible  component of  the Zariski closure of $\Ga_\ga$ coincides with $\mathbf G_0.$
On the other hand, the centralizer   $\mathbf G_\ga$  of $\ga$ in $\mathbf G$ is clearly a Zariski-closed subgroup 
of $\mathbf G$.  It follows that the irreducible  component of $\mathbf G_\ga$ contains (in fact coincides
with)  $\mathbf G_0$ and hence 
$$\Ga_0=  \mathbf G_0\cap \Ga\subset \mathbf G_\ga\cap \Ga= \Ga_\ga.$$

As a consequence,  we see that   $\Ga_0$ acts trivially on $\GaFC$ and hence $\Ad(\Ga)|_{\GaFC}$ is a finite group.
In particular,  $\Ga_0\cap  \GaFC $ is contained in the center $Z(\GaFC)$ of $\GaFC$;
 so $Z(\GaFC)$ has finite index in $\GaFC$ which is therefore a central group.
This proves Items (i) and (ii) of Corollary~\ref{Theo-Plancherel-Linear}. Item (iii)
follows from Theorem~\ref{Pro-PlancherelFC}.ii.

 \vskip.5cm
 Assume now that $\mathbf G$ is connected, that is $\mathbf G=\mathbf G_0.$ 
 Then  $\Ga=\Ga_0$ acts trivially on $\GaFC$ and so  $\GaFC$ coincides with the center $Z(\Ga)$ of 
 $\Ga.$ This proves the first part of Corollary~\ref{Cor-Theo-Plancherel-Linear}.
 
 It remains to prove that the assumption $\mathbf G=\mathbf G_0$ is satisfied in Cases (i) and (ii) of 
 Corollary~\ref{Cor-Theo-Plancherel-Linear}:
  \begin{itemize}
 \item[(i)] Let $\mathbf G$ be  a connected linear algebraic group  over a  countable field $\kk$ of characteristic 0. 
 Then $\Ga= \mathbf G(\kk)$ is Zariski dense in $\mathbf G$, by \cite[Corollary p.44]{Rosenlicht}).
 \item[(ii)] Let $\mathbf G$ be a  connected linear algebraic group $\mathbf G$
 over  a local field $\kk$. Assume that $\mathbf{G}$ has no proper   $\kk$-subgroup $\mathbf{H}$
  such that $(\mathbf G/\mathbf H)(\kk)$ is compact. Then every lattice $\Ga$ in $\mathbf G(\kk)$ is Zariski dense in 
  $\mathbf G$, by  \cite[Corollary 1.2]{Shalom}.
 \end{itemize}

\section{The Plancherel Formula for some countable groups}
\label{Examples}

\subsection{Restricted direct product of finite groups}
\label{Exa-Restricted}
Let $(G_n)_{n\ge 1}$ be a sequence of  finite groups.
Let $\Ga=\prod_{n\ge 1}' G_n$ be the \textbf{restricted direct product} of the $G_n$'s, that is, 
$\Ga$ consists of the sequences $(g_n)_{n\ge 1}$ with $g_n\in G_n$ for all $n$
and $g_n\ne e$ for at most  finitely many $n.$ It is clear that 
$\Ga$ is an FC-group.

Set  $X_n:=\Char(G_n)$ for $n\ge 1$ and let $X=\prod_{n\ge 1} X_n$
be the cartesian product  equipped with the product topology, where 
each $X_n$ carries the discrete topology.
Define a map $\Phi: X\to \Tr(\Ga)$ by 
$$
\Phi((t_n)_{n\ge 1})((g_n)_{n\ge 1})= \prod_{n\ge 1} t_n(g_n) \tout (t_n)_{n\ge 1}\in X, (g_n)_{n\ge 1}\in \Ga
$$
(observe that this product  is well-defined, since  $g_n=e$ and hence $t_n(g_n)=1$ for almost every $n\ge 1$).
Then $\Phi(X)= \Char(\Ga)$ and $\Phi: X\to \Char(\Ga)$ is a homeomorphism (see \cite[Lemma 7.1]{Mautner2})

For every $n\ge 1,$ let $\nu_n$ be the  measure on $X_n$ given by 
$$ 
\nu_n(\{t\})= \dfrac{d_t^2}{\# G_n} \tout t\in X_n,
$$
where $d_t$ is the dimension of the irreducible representation of $G_n$ with $t$ as 
character; observe that $\nu_n$ is a probability measure, since $\sum_{t\in X_n} d_t^2= \# G_n$.

Let $\nu=\otimes_{n\ge 1} \nu$ be the product measure on the Borel subsets of $X.$
The Plancherel measure  on $\Ga$ is the image of $\nu$ under $\Phi$ (see Equation (5.6) in \cite{Mautner2}).
The regular representation $\la_\Ga$ is of type II if and only if infinitely many $G_n$'s are non abelian
(see  \emph{loc.cit.}, Theorem 1  or Theorem~\ref{Theorem-Thoma-Kaniuth} below).

\subsection{Infinite dimensional Heisenberg group}
\label{Exa-Heisenberg}
Let  $\FF_p$ be the field of order $p$ for an odd prime $p$
and let $V =\oplus_{i\in \NN} \FF_p$ be a vector space over $\FF_p$ of countable infinite dimension.
Denote by  $\omega$ the symplectic form on $V \oplus V$ given by 
$$
\omega((x,y),(x',y')) \, = \, \sum_{i \in \NN} (x_iy'_i -y_ix'_i)
\hskip.5cm \text{for} \hskip.2cm 
(x,y), (x',y') \in V \oplus V.
$$
The ``infinite dimensional" Heisenberg group 
over $\FF_p$ is   the group $\Ga$ with underlying set $V \oplus V \oplus \FF_p$
and with multiplication defined by
$$
(x,y,z) (x',y',z') \, = \, (x+x', y+y', z+z'+ \omega((x,y),(x',y')))
$$
for $(x,y,z), (x',y',z') \in \Ga$.

The group $\Ga$ is an FC-group; since $p\geq 3,$ its center $Z$
coincides with $[\Ga, \Ga]$ and consists of the elements of the form $(0, 0, z)$ for $z\in \FF_p.$
Observe  that $\Ga$ is not virtually abelian.

Let $z_0$ be a generator for the cyclic group $Z$ of order $p$.
The unitary dual $\widehat Z$  consists of the  characters 
defined by $\chi_\omega(z_0^j) = \omega^j$ for $j \in \{0, 1, \hdots, p-1\}$ and $\omega\in C_p,$
where $C_p$ is  the group of $p$-th roots of unity in $\CCC.$
\par

For $\omega \in C_p$, the subspace
$$
\Hi_\omega \, = \, \{ f \in \ell^2(\Gamma) \mid f(z_0x) = \omega f(x)
\hskip.2cm \text{for every} \hskip.2cm
x \in \Gamma\}.
$$
is left and right translation invariant and we have an orthogonal decomposition of 
 $\ell^2(\Gamma) =  \bigoplus_{\omega \in C_p} \Hi_\omega.$
The orthogonal projection $P_\omega$ on $\Hi_\omega$ belongs to the center of $L(\Ga)$ and is given by 
$$
P_\omega (f) (x) \, = \, \dfrac{1}{p} \sum_{i = 0}^{p-1} \omega^{-i}f(z_0^i x)
\tout f \in \ell^2(\Gamma), x \in \Gamma.
$$
One checks that $\Vert P_\omega(\delta_e)\Vert^2= 1/p.$

Let $\pi_\omega$ be the restriction of $\lambda_\Gamma$ to $\Hi_\omega$.
Observe that $\H_1$ can be identified with $\ell^2(\Gamma/Z)$ and $\pi_1$ with  $\lambda_{\Gamma/Z}$.

For $\omega \ne 1$, the representation  $\pi_\omega$ is factorial of type II$_1$
and the corresponding character is $\widetilde{\chi_\omega}$ (for more details, see the proof
of  Theorem 7.D.4 in \cite{BH}).

The integral decomposition of $\delta_e$   is 
$$
\delta_e= \dfrac{1}{p}\int_{\widehat{\Gamma/Z}} \chi  d\nu (\chi)+\dfrac{1}{p}
\sum_{\omega \in C_p\setminus \{1\} } \widetilde{\chi_\omega},
$$
with the corresponding Plancherel formula given for every $f\in \CCC[\Ga]$ by
$$
\Vert f\Vert_2^2= \dfrac{1}{p}\int_{\widehat{\Gamma/Z}} |\mathcal F(P_1(f)) (\chi)|^2  d\nu (\chi) +\dfrac{1}{p}\sum_{\omega \in C_p\setminus \{1\} }\sum_{j=0}^{p-1} (f^*\ast f)(z_0^j)\omega^j,
$$
where $\nu$ is the normalized Haar measure of the compact abelian group $\widehat{\Gamma/Z}$
and $\mathcal F$ the Fourier transform.
In particular, $\lambda_\Ga(\Ga)''$  is a direct sum of an abelian von Neumann algebra
and $p-1$ factors of type II$_1$.
For a more general result, see \cite[Theorem 2]{Kaplansky}.

\subsection{An example involving $SL_d(\ZZ)$}
\label{Exa-SLdZ}
Let  $\Lambda=SL_d(\ZZ)$ for an odd integer $d.$
Fix a prime $p$ and for   $n\geq 1,$ let $G_n=SL_3(\ZZ/p^n \ZZ)$, viewed as (finite) quotient of $\La.$
Let $\Ga$ be the semi-direct product $\Lambda \ltimes \prod_{n\ge 1}' G_n$,
where $\Lambda$ acts diagonally in the natural way on the restricted direct product $G:=\prod_{n\ge 1}' G_n$
of the $G_n$'s. 

Since $\La$ is an ICC-group, it is clear that $\GaFC= G.$
The group $K_\Ga$ as in Theorem~\ref{Pro-PlancherelFC} can be identified with the  projective limit 
of the groups $G_n$'s, that is, with $SL_d(\ZZ_p),$ where $\ZZ_p$ is the ring of $p$-adic integers.
Since $\La$ acts trivially on $\Char(G),$ the same is true for the action of $K_\Ga$  on $\Char(G).$ 

Let  $\la_{G}= \int^\oplus_{\Char(G)} \pi_t d\nu(t)$ be the  Plancherel decomposition 
of $\la_G$ (see Example~\ref{Exa-Restricted}). It follows from Theorem~\ref{Pro-PlancherelFC} that the Plancherel decomposition  of $\la_\Ga$ is
$$\la_{\Ga}= \int^\oplus_{\Char(G)} \ind_{G}^\Ga\pi_t d\nu(t).$$

\appendix
\section{Proof of Theorem~\ref{Theorem-Thoma-Kaniuth}}
\label{S: ThomaKaniuth}
\subsection{Easy implications} 
The implications  $(iv')\Rightarrow (iv)$ and $(i)\Rightarrow (iii)$ are obvious;
if $(iv)$ holds then  $\Ga$ is a so-called CCR group and so  $(i)$ holds, by a general fact (see \cite[5.5.2]{Dixm--C*}). 

We are going to  show that $(ii)\Rightarrow (iv'), $
Assume that $\Gamma$ contains an abelian normal subgroup $N$ of finite index.
 Let $(\pi, \H)$ be an irreducible  representation of $\Gamma$.
 
Denote by  $\mathcal B$ the set of Borel subsets of the dual group $\widehat{N}$
and by $\Proj(\H)$ the set of orthogonal projections in $\L(\H).$
Let $E\colon \mathcal B(\widehat{N}) \to \Proj(\Hi)$ be the projection-valued measure on $\widehat N$ associated with the restriction  $\pi \vert_N$
by the SNAG Theorem (see \cite[D.3.1]{BHV}); so, we have 
$$
\pi(n) = \, \int_{\widehat G} \chi(n) dE (\chi)
\hskip.5cm \text{for all} \hskip.2cm
n \in N.
$$

The dual action of $\Ga$  on $\widehat{N}$,
given by $\chi^\ga(n)= \chi(\ga^{-1} n\ga)$ for $\chi\in \widehat{N}$
and $\ga\in \Ga,$ 
factorizes through $\Ga/N$. Moreover, the following covariance relation holds
$$
\pi(\ga)E(B)\pi(\ga^{-1}) \, = \, E(B^{\ga})
\tout
B \in \mathcal B (\widehat N),
$$
where $B^\ga= \{\chi^\ga \mid \chi \in B\}$.

Let $S\in \mathcal B( \widehat{N})$ be the support of $E,$
that is, $S$ is the complement of the largest open subset $U$ of $\widehat{N}$  with $E(U) = 0$. We claim that $S$ consists of a  single $\Ga$-orbit.

Indeed, let $\chi_0\in S$ and let $(U_n)_{n\ge 1}$ be a sequence of
open neighbourhoods of $\chi$ with $\bigcap_{n\ge 1} U_n= \{\chi_0\}.$
Fix $n\ge 1.$ The set $U_n^\Ga$ is  $\Ga$- invariant and hence $E(U_n^\Ga)\in \pi(\Ga)'$, by the covariance relation.
Since $\pi$ is irreducible and $E(U_n)\ne 0,$ we have therefore $E(U_n^\Ga)=I_{\H}.$
 By the usual properties of a projection-valued measure, this implies
that 
$$
E(\chi_0^\Ga)=E(\bigcap_{n\ge 1} U_n^\Ga)=I_{\H}.
$$
and the claim is proved.
 
 Since $S$ is finite, we have
$\Hi = \bigoplus_{\chi \in S} \Hi^{\chi}$,
where 
$$
\Hi^{\chi} \, := \, \left\{ \xi \in \Hi \mid \pi(n) \xi = \chi(n) \xi 
\tout
n \in N \right\};
$$
moreover, since $N$ is a normal subgroup, we have
$\pi(\gamma) \Hi^{\chi} = \Hi^{\chi^{\gamma}}$ for every $\chi \in S$
and every $\gamma\in \Gamma$.
 
 Let $H$ be the stabilizer of $\chi_0$
and let $T\subset \Ga$ be a set of representatives of the right $T$-cosets of $H$.
Then $\Hi^{\chi_0}$ is invariant under $\pi(H)$ and we have 
 $$
\Hi \, = \, \bigoplus_{t \in T} \Hi^{\chi_0^{t}} \, = \, \bigoplus_{t \in T} \pi(t) \Hi^{\chi_0}.
$$
This shows  that $\pi$ is equivalent to the induced representation 
$\Ind_{H}^\Gamma \sigma$,
where $\sigma$ is the subrepresentation of $\pi \vert_{H}$
defined on $\Hi^{\chi_0}$
\par

We claim that $\Ind_{H}^\Gamma \sigma$ is contained in $\Ind_{N}^\Gamma \chi_0$.
Indeed, as is well-known (see \cite[E.2.5]{BHV}), $\Ind_N^{H}(\sigma \vert_N)$
is equivalent to the tensor product representation $\sigma \otimes \lambda_{H/N}$,
where $\lambda_{H/N}$ is the quasi-regular representation on $\ell^2(\Ga/H).$
Since $H/N$ is finite, $1_{H}$ is contained in $\lambda_{H/N}$
and therefore $\sigma$ is contained  in $\Ind_N^{H} (\sigma\vert_N)$.
Notice that $\sigma\vert_N$ is a multiple $n \chi_0$ of $\chi_0$,
for some cardinal $n$.  We conclude that $\pi=\Ind_{H}^\Gamma\sigma$ is contained in $\Ind_{H}^\Gamma (\Ind_N^{H}n\chi_0)=n\Ind_N^\Gamma\chi_0$. Since $\pi$ is irreducible, it follows
that $\pi$ is contained in $\Ind_N^\Gamma\chi_0$.

Now,  $\Ind_{N}^\Gamma \chi_0$ has dimension $[\Ga: N]$;
hence,  $\dim \pi \leq [\Ga: N]$ and so $(iv')$ holds.

\subsection{Proof of the other implications}
We have to give the proof of  the  implication  $(iii)\Rightarrow (i)$ and the equivalence $(v)\Leftrightarrow (iv).$

In the sequel, $\Ga$ will be a countable group
and $\la_\Ga=\int_{\Char(\Ga)}^\oplus \pi_t d\mu(t)$  the direct integral 
decomposition given by the Plancherel Theorem~\ref{Pro-Plancherel}.
Recall (see Section~\ref{ProofPL}) that, if we write $\delta_e= \int_{\Char(\Ga)}^\oplus \xi_t d\mu(t)$,
then $\xi_t$ is a  cyclic vector in the Hilbert space   $\H_t$ of $\pi_t$ 
and $t=\langle \pi_t(\cdot) \xi_t\mid \xi_t\rangle$, for $\mu$-almost every $t\in \Char(\Ga).$

\subsubsection{Case where the FC-centre of $\Ga$ has infinite index}
We assume  that $[\Gamma\colon \Ga_{\rm fc}]$ is infinite; we claim that $\la_\Ga$ is of type II.

By Remark~\ref{SupportPlancherel}.i, there exists a subset 
$X$ of $\Char(\Ga)$ with $\mu(X)=1$ such that $t=0$
outside $\Ga_{\rm fc}$ for every $t\in X.$

Let $t\in X$.  Then the  factor $\pi_t(\Gamma)''$ is infinite dimensional.
Indeed, since  $[\Gamma \colon \Ga_{\rm fc}]$ is infinite, we can find a sequence $(\ga_n)_{n\geq 1}$
in $\Ga$ with $\ga_m^{-1} \ga_n\notin \Ga_{\rm fc}$ for every $m,n$ with 
$m\ne n.$ Then
$$
\langle \pi_t(\ga_n) \xi_t\mid \pi_t(\ga_m)\xi_t\rangle= 
\langle \pi_t(\ga_m^{-1}\ga_n) \xi_t\mid\xi_t\rangle= t(\ga_m^{-1}\ga_n)=0,
$$
for $m\ne n$; so, $ (\pi_t(\ga_n)\xi_t)_{n\ge 1}$ is an orthonormal sequence in $\H_t.$
This implies that $ (\pi_t(\ga_n))_{n\ge 1}$ is a linearly independent sequence
in $\pi_t(\Ga)''$ and the claim is proved.

Observe that we have proved, in particular, that $\Ga$ is not of type I.

\subsubsection{Reduction to FC-groups}
 \label{SS:Reduction}
Let $\Char(\Ga)_{\rm{fd}}$  be the set of $t\in \Char(\Ga)$ for which 
$\pi_t(\Ga)''$ is finite dimensional.  Then 
$$\Char(\Ga)_{\rm{fd}}= \bigcup_{n\geq 1} \Char(\Ga)_{n},$$
where $\Char(\Ga)_{n}$ is the set of $t$ such that 
$\dim \pi_t(\Ga)''=n$. 
We claim that  $\Char(\Ga)_{n}$ and hence $\Char(\Ga)_{\rm{fd}}$ is a measurable subset of  $\Char(\Ga)$.

Indeed, let $\mathcal F$ be collection of
finite subsets of $\Ga.$
 For every $F\in \mathcal F,$ 
let $C_F$ be the set of $t\in \Char(\Ga)$ such that the family
$(\pi_t(\ga))_{\ga\in F}$ is linearly independent, equivalently (see Remark~\ref{SupportPlancherel}.ii),
such that $(\pi_t(\ga)\xi_t)_{\ga\in F}$ is linearly independent.
Since $t=\langle \pi_t(\cdot) \xi_t\mid \xi_t\rangle,$ it follows that
$$
C_F:=\left\{ t\in \Char(\Ga) \mid \det (t(\ga^{-1}\ga'))\neq 0 \tout (\ga, \ga')\in F\times F \right\}
$$
and this shows that $C_F$ is measurable. Since
$$
\Char(\Ga)_n= \bigcup_{F\in \mathcal F: \#F=n} C_F \setminus  \left(\bigcup_{F'\in \mathcal F: \# F'>n} C_{F'} \right)\leqno{(*)}
$$
and $\mathcal F$ is countable, it follows that $\Char(\Ga)_n$ is measurable.

Assume  now that $[\Gamma\colon \Ga_{\rm fc}]$ is finite.
Observe that, for a cyclic representation $\pi$ of $\GaFC,$ the induced representation $\Ind_{\GaFC}^\Ga \pi$ is cyclic
and so $(\Ind_{\GaFC}^\Ga \pi)(\Ga)''$ is finite dimensional 
if and only if $\pi(\GaFC)''$ is finite dimensional.

let $\nu$ be the Plancherel measure of $\GaFC.$
It follows from Theorem~\ref{Pro-PlancherelFC} that $\mu(\Char(\Ga)_{\rm{fd}})= \nu(\Char(\GaFC)_{\rm{fd}})$; in particular,
we have $\mu(\Char(\Ga)_{\rm{fd}})= 0$ (or $\mu(\Char(\Ga)_{\rm{fd}})= 1$)  if and only if $ \nu(\Char(\GaFC)_{\rm{fd}})= 0$ (or $\nu(\Char(\GaFC)_{\rm{fd}})= 1$), that is,   
$\la_\Ga$ is  of type I (or of type II) if and only if $\la_{\GaFC}$ is of type I (or of type II).

 Observe also that $\Ga$ is virtually abelian if and only if 
 $\GaFC$ is virtually abelian. As a consequence, we see that it suffices to 
prove the  implication  $(iii)\Rightarrow (i)$ and the equivalence $(v)\Leftrightarrow (iv)$ 
in the case where $\Ga= \GaFC.$
 
\subsubsection{Case of an  FC-group}

  We will need the following lemma of independent interest, which is valid for an arbitrary countable group
  $\Ga.$ Let $r:\Char(\Ga)\to \Tr([\Ga, \Ga])$ be the restriction map.
  We will identify $\Char(\Ga/[\Ga, \Ga])$  with the set $\{s\in\Char(\Ga) \mid r(s)=1_{[\Ga, \Ga]}\},$
  that is, with the set of unitary characters  of $\Ga.$
  Observe that, for every 
  $s\in \Char(\Ga/[\Ga, \Ga])$ and $t\in \Char(\Ga),$ we have $st\in \Char(\Ga).$

  \begin{lemma}
   \label{Lem-CommutatorChar}  
  Let $\Ga$ be a countable group and $t, t'\in \Char(\Ga)$ be such that $r(t)=r(t').$ Then 
 there exists $s\in \Char(\Ga/[\Ga, \Ga])$ such that $t'=st.$
 \end{lemma}
 \begin{proof} The integral decomposition of  $\mathbf{1}_{[\Ga, \Ga]}\in \Tr(\Ga)$ into 
  characters is given by
  $$\mathbf{1}_{[\Ga, \Ga]}=\int_{\Char(\Ga/[\Ga, \Ga])} s d\nu (s),$$
  where $\nu$ is the Haar measure of  $\Ga/[\Ga, \Ga].$
  By assumption, we have $t\mathbf{1}_{[\Ga, \Ga]}= t'\mathbf{1}_{[\Ga, \Ga]}$ and hence
  $$
 t\mathbf{1}_{[\Ga, \Ga]}= \int_{\Char(\Ga/[\Ga, \Ga])} ts d\nu (s)=\int_{\Char(\Ga/[\Ga, \Ga])} t's d\nu (s).
  $$
 By uniqueness of  integral decomposition, it follows that the images $\nu_t$ and $\nu_{t'}$ of $\nu$ under 
 the maps $\Char(\Ga/[\Ga, \Ga])\to \Char(\Ga)$ given respectively by multiplication with $t$ and $t'$ coincide.
 In particular, the supports of $\nu_t$ and $\nu_{t'}$ are the same, that is, 
 $t\Char(\Ga/[\Ga, \Ga])= t'\Char(\Ga/[\Ga, \Ga])$ and the claim follows.
\end{proof}

\vskip.2cm
We assume from now on that $\Gamma=\GaFC.$

 \vskip.5cm
 \textbf{Step 1}  We claim that the regular representation $\la_\Ga$ is  of type II
  if and only if $[\Ga, \Ga]$ is infinite.

We have to show that  $\mu(\Char(\Ga)_{\rm{fd}})>0$ if and only if $[\Ga, \Ga]$ is finite.

Assume first that $[\Ga, \Ga]$ is finite.  The representation $\la_{\Ga/[\Ga, \Ga]}$,
lifted to $\Ga$, is a subrepresentation of $\la_\Ga,$ since
$\ell^2(\Ga/[\Ga, \Ga])$ can be viewed in an obvious way as $\Ga$-invariant 
subspace of $\ell^2(\Ga)$. As $\Ga/[\Ga, \Ga]$ is abelian,  $\la_{\Ga/[\Ga, \Ga]}$is of type
I and so $\mu(\Char(\Ga)_{\rm{fd}})>0.$

Conversely, assume that $\mu(\Char(\Ga)_{\rm{fd}})>0.$
Since $\Ga$ is an FC-group,
it suffices to show that $\Ga$ has a subgroup of finite index 
with finite commutator subgroup (see \cite[Lemma 4.1]{Neumann}).

As $\mu(\Char(\Ga)_{\rm{fd}})>0$ and $\Char(\Ga)_{\rm{fd}}= \bigcup_{n\geq 1}\Char(\Ga)_{n},$ we have  $\mu(\Char(\Ga)_{n})>0$ for some $n\ge 1.$
It follows from (*) that  there exists $F\in \mathcal F$ with $|F|=n$
such that $\mu(C_F \cap \Char(\Ga)_{{n}})>0.$

Let $\Lambda$ be the subgroup of $\Ga$ generated by $F.$
Since $\Ga$ is an FC-group and $\Lambda$ is finitely generated, the centralizer 
$H:=\Cent_\Ga(\Lambda)$ of $\Lambda$ in $\Ga$ has finite index.

Let $t\in C_F \cap \Char(\Ga)_{\rm{n}}$ and $\ga_0\in H$.  On the one hand, since $(\pi_t(\ga))_{\ga\in F}$ is a basis of the vector space $\pi_t(\Ga),$ we have 
$\pi_t(\Lambda)''= \pi_t(\Ga)''.$ On the other hand, as $\ga_0$ centralizes $\Lambda,$ we have $\pi_t(\ga_0)\in \pi_t(\Lambda)'$. Hence, $\pi_t(\ga_0)$ belongs to the center $\pi_t(\Ga)'\cap \pi_t(\Ga)''$ of the factor
$\pi_t(\Ga)''$ and so $\pi_t(\ga_0)$ is a scalar multiple of  $I_{\H_t}$.
It follows in particular  that $\pi_t$ is trivial on $[H,H].$
 As a result, the subrepresentation 
$\int_{C_F \cap \Char(\Ga)_{\rm{n}}}^\oplus \pi_t d\mu(t)$
of $\la_\Ga$ is trivial on $[H,H].$  Since the matrix coefficients of $\la_\Ga$ vanish
at infinity,  it follows that $[H,H]$ is finite and the claim is proved.

\vskip.5cm
 In view of what we have shown so far, we  may and will assume from now on  that $[\Ga, \Ga]$ is finite and
 that $\la_\Ga$ is of type I.  We are going  to show that $\Ga$ is a virtually abelian
 (in fact, a central) group and this will finish the proof of Theorem~\ref{Theorem-Thoma-Kaniuth}.
 
 \vskip.3cm
 Set $N:= [\Ga, \Ga]$ and let  $r:\Char(\Ga)\to \Tr(N)$ be the restriction map.
  
 \vskip.5cm
 \textbf{Step 2} We claim that there exist  finitely many functions $s_1, \dots, s_m$ in $\Tr(N)$
 such that $r(t)\in \{s_1, \dots, s_m\}$, for $\mu$-almost every $t\in \Char(\Ga).$
 
 Indeed, since $\la_\Ga$ is of type I, there exists a subset $X$ of $\Char(\Ga)$ with 
 $\mu(X)=1$ such that $\pi_t(\Ga)''$ is finite dimensional for  every  $t\in X.$ 
 
 Let $t\in X.$ The Hilbert space $\H_t$ of $\pi_t$ is finite dimensional and $\pi_t$ is a (finite) multiple of
 an irreducible representation $\sigma_t$ of $\Ga.$ As $\pi_t$ and $\sigma_t$ have the
 same normalized character, we  may assume that $\pi_t$ is irreducible.

 Let $\K$ be an irreducible $N$-invariant subspace of $\H_t$ and let $\rho$ be the corresponding 
 equivalence class of representation of $N.$ 
 For $g\in \Ga,$ the subspace $\pi_t(g)\K$ is $N$-invariant with $\rho^g$ as   corresponding representation of 
 $N$. Since $\pi_t$ is irreducible, we have $\H_t= \sum_{g\in \Ga} \pi_t(g) \K.$
 
 Let $L$ be the stabilizer of $\rho$; observe that $L$ has finite index in $\Ga,$ 
 since $L$ contains the centralizer of $N$ and $\Ga$ is an FC-group.
  Let $g_1, \dots, g_r$  be a set of representatives for the coset space $\Ga/L$.   Then $\H_t= \oplus_{j=1}^r \pi_t(g_j) \K_{\rho},$
 where $\K_{\rho}$ is the sum of all $N$-invariant subspaces of $\H_t$ with corresponding representation equivalent to $\rho.$
 The normalized trace of the representation of $N$ on $\pi_t(g_j) \K_{\rho}$ is $\chi_\rho^{g_j}$, where 
 $\chi_\rho$ is the normalized character of $\rho.$ It follows that, for every $g\in N,$ we have
 $$
 t(g) = \dfrac{1}{r}\sum_{j=1}^r \chi_\rho^{g_j}(g).
 $$
 Since $[\Ga, \Ga]$ is finite, $[\Ga, \Ga]$  has only finitely many  equivalence classes of irreducible representations and
 the claim follows.

 \vskip.5cm
 \textbf{Step 3} We claim that the center $Z(\Ga)$ has finite index in $\Ga$.
 
 Indeed, by Step 2, there exists a subset $X$ of $\Char(\Ga)$ with 
 $\mu(X)=1$ and finitely many $t_1, \dots, t_m\in X$ such that 
 $r(t)\in \{r(t_1), \dots, r(t_m)\}$ and such that $\H_t$ is finite dimensional  for every $t\in X.$
 
 It follows from Lemma~\ref{Lem-CommutatorChar} that, for every $t\in X$,  there exists  $s\in \Char(\Ga/[\Ga, \Ga])$
 such that $t= st_i$ for some $i\in \{1, \dots, m\}$ 
 and hence $\dim \pi_t(\Ga)''=\dim \pi_{t_i}(\Ga)''$.
 As a result,  we can find a finitely generated normal subgroup $M$ of $\Ga$ such that
$\dim \pi_t(\Ga)''=\dim \pi_{t}(M)''$ for every $t\in X.$

Since  the centralizer $C$ of $M$ in $\Ga$ has finite index,
it suffices to show that $C$ contains $Z(\Ga).$ 

For  $g\in C, \ga\in \Ga$ and $x\in M,$ we have
$t( x^{-1}g\ga g^{-1})= t( x^{-1}\ga )$, that is, 
$$\langle \pi_t(g\ga g^{-1} )\xi_t\mid \pi_t(x)\xi_t\rangle=\langle \pi_t(\ga )\xi_t\mid \pi_t(x)\xi_t\rangle.
$$
Since $\dim \pi_t(\Ga)''=\dim \pi_{t}(M)''$, 
the linear span of $\pi_t(M)\xi_t$ is dense in $\H_t$ and
this implies that $\pi_t(g\ga g^{-1} )\xi_t= \pi_t (\ga)\xi_t$ for  all $g\in C, \ga\in \Ga,$ and $t\in X.$
It follows that $\la_\Ga(g\ga g^{-1}) \delta_e= \la_\Ga(\ga)\delta_e$
and hence $g\ga g^{-1} =\ga$ for  all $g\in C$ and $ \ga\in \Ga$; so,  $C \subset Z(\Ga)$.


\begin{thebibliography}{99}

\bib{BHV}{book}{
   author={Bekka, B.},
   author={ Harpe, P. de la},
   author={Valette, A.},
   title={Kazhdan's property (T)},
   series={New Mathematical Monographs},
   volume={11},
   publisher={Cambridge University Press, Cambridge},
   date={2008},
}


\bib{BH}{book}{
   author={Bekka, B.},
   author={ Harpe, P. de la},
   title={Unitary representations of groups,
duals, and characters},
series={Graduate Studies in Mathematics},
   publisher={American Mathematical Society, Providence, RI},
}


\bib{C-PJ}{article}{
   author={Corwin, L.},
   author={Pfeffer Johnston, C.},
   title={On factor representations of discrete rational nilpotent groups
   and the Plancherel formula},
   journal={Pacific J. Math.},
   volume={162},
   date={1994},
   number={2},
   pages={261--275},
}

	


\bib{Dixm--C*} {book}{
    author = {Dixmier, J.},
     title = {{$C\sp*$}-algebras},
 publisher = {North-Holland Publishing Co., Amsterdam-New York-Oxford},
      year= {1977},
}

 \bib{Dixm--vN}{book}{
   author={Dixmier, J.},
   title={Les alg\`ebres d'op\'{e}rateurs dans l'espace hilbertien (alg\`ebres de
   von Neumann)},
   language={French},
   note={Deuxi\`eme \'{e}dition, revue et augment\'{e}e;
   Cahiers Scientifiques, Fasc. XXV},
   publisher={Gauthier-Villars \'{E}diteur, Paris},
   date={1969},
}
\bib{Glimm}{article}{
   author={Glimm, J.},
   title={Type I $C^{\ast} $-algebras},
   journal={Ann. of Math. (2)},
   volume={73},
   date={1961},
   pages={572--612},
}





\bib{Kaniuth}{article}{
   author={Kaniuth, E.},
   title={Der Typ der regul\"{a}ren Darstellung diskreter Gruppen},
   language={German},
   journal={Math. Ann.},
   volume={182},
   date={1969},
   pages={334--339},
}
\bib{Kaplansky}{article}{
   author={Kaplansky, I.},
   title={Group algebras in the large},
   journal={Tohoku Math. J. (2)},
   volume={3},
   date={1951},
   pages={249--256},
}

  
 \bib{Mackey-Borel}{article}{
   author={Mackey, G.W.},
   title={Borel structure in groups and their duals},
   journal={Trans. Amer. Math. Soc.},
   volume={85},
   date={1957},
   pages={134--165},
}
 
 \bib{Mackey0}{article}{
   author={Mackey, G. W.},
   title={Induced representations of locally compact groups. I},
   journal={Ann. of Math. (2)},
   volume={55},
   date={1952},
   pages={101--139},
}
\bib{Mackey2}{article}{
   author={Mackey, G.W.},
   title={Induced representations and normal subgroups},
   conference={
      title={Proc. Internat. Sympos. Linear Spaces},
      address={Jerusalem},
      date={1960},
   },
   book={
      publisher={Jerusalem Academic Press, Jerusalem; Pergamon, Oxford},
   },
   date={1961},
   pages={319--326},
}
\bib{Mautner}{article}{
   author={Mautner, F. I.},
   title={The structure of the regular representation of certain discrete
   groups},
   journal={Duke Math. J.},
   volume={17},
   date={1950},
   pages={437--441},
   }
   
  \bib{Mautner2}{article}{
   author={Mautner, F. I.},
   title={The regular representation of a restricted direct product of
   finite groups},
   journal={Trans. Amer. Math. Soc.},
   volume={70},
   date={1951},
   pages={531--548},
}

   

\bib{Neumann}{article}{
   author={Neumann, B. H.},
   title={Groups with finite classes of conjugate subgroups},
   journal={Math. Z.},
   volume={63},
   date={1955},
   pages={76--96},
}

\bib{PJ}{article}{
   author={Pfeffer Johnston, C.},
   title={On a Plancherel formula for certain discrete, finitely generated,
   torsion-free nilpotent groups},
   journal={Pacific J. Math.},
   volume={167},
   date={1995},
   number={2},
   pages={313--326},
}




\bib{Rosenlicht}{article}{
   author={Rosenlicht, M.},
   title={Some rationality questions on algebraic groups},
   journal={Ann. Mat. Pura Appl. (4)},
   volume={43},
   date={1957},
   pages={25--50},
}


\bib{Sakai}{book}{
   author={Sakai, S.},
   title={$C\sp*$-algebras and $W\sp*$-algebras},
   note={Ergebnisse der Mathematik und ihrer Grenzgebiete, Band 60},
   publisher={Springer-Verlag, New York-Heidelberg},
   date={1971},
   pages={xii+253},
   review={\MR{0442701}
   },
}
\bib{Segal}{article}{
   author={Segal, I. E.},
   title={An extension of Plancherel's formula to separable unimodular
   groups},
   journal={Ann. of Math. (2)},
   volume={52},
   date={1950},
   pages={272--292},
}
\bib{Shalom}{article}{
   author={Shalom, Y.},
   title={Invariant measures for algebraic actions, Zariski dense subgroups
   and Kazhdan's property (T)},
   journal={Trans. Amer. Math. Soc.},
   volume={351},
   date={1999},
   number={8},
   pages={3387--3412},
}


\bib{Thoma1}{article}{
   author={Thoma, E.},
   title={\"{U}ber unit\"{a}re Darstellungen abz\"{a}hlbarer, diskreter Gruppen},
   language={German},
   journal={Math. Ann.},
   volume={153},
   date={1964},
   pages={111--138},
}

  \bib{Thoma2}{article}
  {
 AUTHOR = {Thoma, E.},
       TITLE = {Eine {C}harakterisierung diskreter {G}ruppen vom {T}yp {I}},
   JOURNAL = {Invent. Math.},
  FJOURNAL = {Inventiones Mathematicae},
    VOLUME = {6},
      YEAR = {1968},
     PAGES = {190--196},
      ISSN = {0020-9910},
   MRCLASS = {22.60 (46.00)},
  MRNUMBER = {0248288},
MRREVIEWER = {L. Greenberg},
}


\bib{Thoma-Plancherel}{article}{
   author={Thoma, E.},
   title={\"{U}ber das regul\"{a}re Mass im dualen Raum diskreter Gruppen},
   language={German},
   journal={Math. Z.},
   volume={100},
   date={1967},
   pages={257--271},
}
		



 
\bib{Zimmer}{book}
 {AUTHOR = {Zimmer, R. J.},
     TITLE = {Ergodic theory and semisimple groups},
    SERIES = {Monographs in Mathematics},
    VOLUME = {81},
 PUBLISHER = {Birkh\"{a}user Verlag, Basel},
      YEAR = {1984},
     PAGES = {x+209},
     }


\end{thebibliography}
\end{document}